\documentclass[12pt]{extarticle} 
\usepackage{amsfonts}   
\usepackage{amsmath}    
\usepackage{amsthm}     
\usepackage{amssymb}    
\usepackage{latexsym}   
\usepackage{bbm}        
\usepackage{tikz}       
\usepackage{graphics}   
\usepackage[a4paper, total={7in, 9in}]{geometry} 
\usepackage[margin=1cm]{caption} 
\usepackage{float}      
\usepackage{enumerate}  
\usepackage{algorithm}  
\usepackage{algorithmic}
\usepackage{color}      
\usepackage[noadjust]{cite} 
\usepackage{authblk}    
\usepackage{setspace}   

\newtheorem{theorem}{Theorem}[section]
\newtheorem{lemma}[theorem]{Lemma}

\newtheorem{proposition}[theorem]{Proposition}
\newtheorem{corollary}[theorem]{Corollary}

\theoremstyle{definition}
\newtheorem{definition}[theorem]{Definition}

\theoremstyle{remark}
\newtheorem{remark}[theorem]{Remark}

\newcommand{\cM}{\mathcal{M}}

\newcommand{\cE}{\mathcal{E}}

\newcommand{\cN}{\mathcal{N}}

\newcommand{\cB}{\mathcal{B}}

\newcommand{\cP}{\mathcal{P}}
\newcommand{\cC}{\mathcal{C}}

\newcommand{\bP}{\mathbb{P}}

\newcommand{\Z}{\mathbb{Z}}

\newcommand{\R}{\mathbb{R}}
\newcommand{\ep}{\epsilon }
\newcommand{\la}{\lambda }

\newcommand{\si}{\sigma }

\newcommand{\ga}{\gamma }
\newcommand{\Ga}{\Gamma }

\newcommand{\ones}{\mathbbm{1}}
\newcommand{\one}{\mathbf{1}}

\newcommand{\Mod}{\operatorname{Mod}}
\newcommand{\Dom}{\operatorname{Dom}}

\newcommand{\Adm}{\operatorname{Adm}}
\newcommand{\Ext}{\operatorname{Ext}}

\newcommand{\co}{\operatorname{conv}}
\newcommand{\cl}{\operatorname{cl}}
\newcommand{\bi}{\begin{itemize}}
\newcommand{\ei}{\end{itemize}}

\newcommand{\lbr}{\left\{ }
\newcommand{\rbr}{\right\} }

\newcommand{\cI}{\mathcal{I}}

\newcommand{\cU}{\mathcal{U}}
\newcommand{\supp}{\operatorname{supp}}
\usepackage{comment}
\numberwithin{equation}{section}
\usepackage[colorinlistoftodos,prependcaption,textsize=tiny]{todonotes}

\newcommand{\KL}{\text{KL}}

\newcommand{\MKL}{\operatorname{MKL}}
\newcommand{\dep}{\operatorname{dep}}

\newcommand{\Gahat}{\widehat{\Ga}}

\usepackage{hyperref}   
\usepackage{tikz}

\usepackage{placeins}

\begin{document}
\title{\bf Base Modulus for Matroid Truncation, Strength, and Fractional Arboricity\thanks{This material is based upon work supported by the National Science Foundation under Grant No. 2154032.}}

\author{Huy Truong\thanks{Corresponding author. E-mail: \texttt{huytruong@ksu.edu}} \quad Pietro Poggi-Corradini}

\affil[1]{\small Dept. of Mathematics, Kansas State University, Manhattan, KS 66506, USA.}

\date{}

\maketitle


\begin{abstract}

In \cite{truong2024modulus}, we provided results on the $p$-modulus of the family of all bases of matroids and showed that it recovers various concepts in matroid theory, including strength, fractional arboricity, and principal partitions. 
In particular, the unique optimal density  $\eta^*$ that arises for $p$-modulus, which we will refer to as universal density from now on, was shown to recover the concept of lexicographical base in polymatroids. Since truncation is a fundamental operation in matroid theory, it is
natural to ask how the universal density behaves under matroid truncation.
  In this paper, we first provide the universal density of every truncation of a given matroid; equivalently, we determine the principal partition for every matroid truncation. Next, we give a new characterization of the universal density using the Kullback--Leibler divergence. 
 Furthermore, we study the notion of strictly homogeneous matroids, generalizing the corresponding notion in graphs from \cite{albin2024minimizing}. We also offer several insights related to strength, fractional arboricity, and give the set of probability mass functions (pmfs) for bases that induce the universal density in a simple case. Finally, this paper also addresses two optimization problems for graph structures, particularly those involving edge-disjoint spanning trees and forest edge-coverings.

\end{abstract}

\noindent {\bf Keywords:} Truncation, matroid, universal base, strength, fractional arboricity.

\vspace{0.1in}

\noindent {\bf 2020 Mathematics Subject Classification:} 90C27 (Primary) ; 05B35 (Secondary)

\section{Introduction}
 The study of spanning trees has been a central topic in graph theory. A well-known result in this area is the spanning tree packing theorem, established independently by Nash-Williams \cite{edge-disjoint} and Tutte \cite{on}, which characterizes graphs with $k$ edge-disjoint spanning trees, for any integer $k > 0$. Another well-known theorem is the forest covering theorem of Nash-Williams \cite{nash1964decomposition}, which describes graphs with $k$ forests whose union equals the set of all edges. 
 Edmonds \cite{edmonds1965lehman} proved the corresponding theorems, establishing a deep connection between graph theory and matroid theory. 
 \vspace{\baselineskip}
 
 Let's recall some notations. Let $M$ be a finite matroid with no loops. Let $r_M$ (or simply $r$, when the context is clear) denote the rank function of $M$. The ground set of $M$ is denoted by $E(M)$, while $\cB(M)$ denotes the collection of bases of $M$. For a set $X$, we write $|X|$  for its cardinality, and if $Y$ is another set, then we write $X-Y$ for the relative complement of $Y$ in $X$.
 Let \( \tau(M) \) represent the maximum number of disjoint bases of \( M \) and let  \( a(M) \) be the minimum number of bases of \( M \) whose union equals \( E(M) \). The following theorem by Edmonds \cite{edmonds1965lehman} are well known.
 	\begin{theorem} \label{edmonds-SD}\cite{edmonds1965lehman} Let $M$ be a matroid with $r(M) > 0$. Each of the following statements holds.
 \begin{enumerate}
 	\item $\tau(M) \geq k$ if and only if $\forall X \subseteq E(M)$, $|E(M) - X| \geq k \left(r(M) - r(X)\right)$.
 	\item $a(M) \leq k$ if and only if $\forall X \subseteq E(M)$, $|X| \leq k \cdot r(X)$.
 \end{enumerate}
 \end{theorem}
 For any subset $X \subseteq E(M)$ with $r(X) > 0$, the {\it density} of $X$ is defined as
 $\theta_M(X) = |X|/r(X).$
When the matroid $M$ is understood from the context, we often omit the subscript $M$. We also use $\theta(M)$ for $\theta(E(M))$. The {\it strength} $S(M)$ and the {\it fractional arboricity} $D(M)$ of $M$ are defined as follows:
 \begin{equation}
 S(M) = \min \left\{ \frac{|X|}{r(E) - r(E-X)} : X \subseteq E, r(E-X) < r(E) \right\},
 \end{equation}
 and
 \begin{equation}
 D(M) = \max \left\{ \theta(X) : X \subseteq E(M), r(X) > 0\right\}.
 \end{equation}
 Theorem \ref{edmonds-SD} indicates that
 $
 \tau(M) = \lfloor S(M) \rfloor,$ and $   a(M) = \lceil D(M) \rceil,
 $
 where $\lfloor\cdot \rfloor$ is the floor function and $\lceil \cdot \rceil$ is the ceiling function.
 The notions of strength and fractional arboricity are two well-known concepts in matroid theory. They are also related to the theory of modulus for bases of matroids. The theory of {\it discrete modulus} on graphs and matroids has been extensively studied in recent years \cite{modulus, pietrofairest,truong2024modulus,truong2025hypertree}. Discrete modulus  is a very flexible and general tool for measuring the richness of families of objects defined on a finite set. One main problem in this theory is the following: 
  
 \begin{equation}\label{eq:2norm}
 	\min \{ \Vert \eta \Vert_p: \eta \in \co(\cB(M))\}
 \end{equation}
 where $\Vert\cdot\Vert_{p}$ is the $p$-norm for $1\leq p \leq \infty$ and $\co(\cB(M))$ is the convex hull of the indicator functions of all bases in $\cB(M)$ of $M$.
 It was shown in \cite{truong2024modulus} that the optimization problem \ref{eq:2norm} obtains the same  unique optimal vector $\eta^*$ for all $1<p < \infty$. In this paper, we call this vector the {\it universal density} of $M$ and denote it by $\eta^*_M$. Later in this paper, we will see that $\eta^*_M$ is also an optimal vector for the case $p=\infty$, see Remark \ref{re:modinf}.
 Furthermore, it was shown in \cite{truong2024modulus} that \[S(M) = \frac{1}{\eta^*_{max}}, \quad D(M) = \frac{1}{ \eta^*_{min}},\] where $\eta^*_{max}$ and $\eta^*_{min}$ are the maximum and minimum values, respectively, of $\eta^*_M$.  Note that the universal density of a matroid determines its principal partition, see \cite{fujishige2009theory} for further details on the
 principal partition.
 \vspace{\baselineskip}
 
 The notion of strength and fractional arboricity can be specialized to graphs. The strength of graphs admits a natural generalization called the {\it $c$-order edge toughness} of a graph, which is defined as
 \begin{equation}
 	\tau_c(G)= \min\limits_{X \subset E: \ \omega(G-X)>c} \left\{ \frac{|X|}{\omega(G-X)-c}\right\},
 \end{equation}
 where $c$ is a natural number and $\omega(G-X) $ is the number of components of the graph $G-X$. The $1$-order edge toughness of a graph is the graph's strength.
 In \cite{onthe}, it was proved that $\tau_c(G) \geq k$ if and only if $G$ has $k$ edge-disjoint spanning 
 forests with exactly $c$ components, thus generalizing the result of Nash-Williams \cite{edge-disjoint} and Tutte \cite{on}. This result was also proved in \cite{the-higher-order} using the notion of matroid truncation. Let us recall this notion. Let $M$ be a matroid and let $t$ denote a fixed integer such that $1 \leq t \leq r(M)$. A $t$-independent set in $M$ is an independent set of size at most $t$. The {\it $ t$-truncation} of $M$ is a matroid on $E$ whose collection of independent sets is the set of all $t$-independent sets of $M$, and it is denoted by $M_t$. In the graph case, let $M(G)$ be the graphic matroid associated with a graph $G$. Then, a spanning forest with exactly $t$ edges of a graph $G$ is a base of the matroid truncation $M(G)_t$, for $1\leq t \leq |V|-1$. Hence, we can study the family of spanning forests with exactly $t$ edges by considering them as bases of $M(G)_t$. The first contribution of this paper is to show that the universal density $\eta^*_{M_t}$ of $M_t$ can be deduced from the universal density $\eta^*_{M}$ of $M$ for any natural number $t\leq r(M)$. Furthermore, we also introduce and study the following family of matroids generated by $M$: For $t=r(M),r(M)+1,\dots,|E|-1$, we define the {\it dual truncation} $M_t$ of $M$ as follows:
 
\begin{equation}
M_t:= ((M^*)_{|E|-t})^*,
\end{equation}where $M^*$ denotes the dual matroid of $M$, with the base family of $M_t$ is 
\begin{equation}
	\cB(M_t)= \left\{ X\subset E: |X|=t, X \supset B \text{ for some base } B \in \cB(M)\right\}.
\end{equation}
By an abuse of notation, we use $M_t$ to denote both the $t$-truncation 
(for $t \leq r(M)$) and the dual truncation (for $t > r(M)$), 
the intended meaning will be clear from the range of $t$.
Specifically, we will also provide the universal densities for the family of dual truncations. See Figure \ref{figure1} and Table \ref{table1} for an example of a graphic matroid $M$ with associated universal densities of $M_t$, for $t=1,\dots,|E|$.

\vspace{\baselineskip}

\begin{figure}[t]
	\centering
	\begin{tikzpicture}[scale=1.3, auto, node distance=3cm, thin]
		\begin{scope}[every node/.style={circle,draw=black,fill=black!100!,font=\sffamily\Large\bfseries}]
			\node (A) [scale=0.3] at (0,0) {};
			\node (B) [scale=0.3] at (1,0) {};
			\node (C) [scale=0.3] at (1.5,0.87) {};
			\node (D) [scale=0.3] at (1,1.73) {};
			\node (E) [scale=0.3] at (0,1.73) {};
			\node (F) [scale=0.3] at (-0.5,0.87) {};
			\node (G) [scale=0.3] at (0.25,0.43) {};
			\node (H) [scale=0.3] at (0.75,0.43) {};
			\node (I) [scale=0.3] at (1,0.87) {};
			\node (J) [scale=0.3] at (0.75,1.3) {};
			\node (K) [scale=0.3] at (0.25,1.3) {};
			\node (L) [scale=0.3] at (0,0.87) {};
			
			\node (A1) [scale=0.3] at (3,0) {};
			\node (B1) [scale=0.3] at (4,0) {};
			\node (C1) [scale=0.3] at (4.5,0.87) {};
			\node (D1) [scale=0.3] at (4,1.73) {};
			\node (E1) [scale=0.3] at (3,1.73) {};
			\node (F1) [scale=0.3] at (2.5,0.87) {};
			\node (G1) [scale=0.3] at (3.25,0.43) {};
			\node (H1) [scale=0.3] at (3.75,0.43) {};
			\node (I1) [scale=0.3] at (4,0.87) {};
			\node (J1) [scale=0.3] at (3.75,1.3) {};
			\node (K1) [scale=0.3] at (3.25,1.3) {};
			\node (L1) [scale=0.3] at (3,0.87) {};
			
			\node (A2) [scale=0.3] at (1.5,-2.6) {};
			\node (B2) [scale=0.3] at (2.5,-2.6) {};
			\node (C2) [scale=0.3] at (3,-1.73) {};
			\node (D2) [scale=0.3] at (2.5,-0.87) {};
			\node (E2) [scale=0.3] at (1.5,-0.87) {};
			\node (F2) [scale=0.3] at (1,-1.73) {};
			\node (G2) [scale=0.3] at (1.75,-2.17) {};
			\node (H2) [scale=0.3] at (2.25,-2.17) {};
			\node (I2) [scale=0.3] at (2.5,-1.73) {};
			\node (J2) [scale=0.3] at (2.25,-1.3) {};
			\node (K2) [scale=0.3] at (1.75,-1.3) {};
			\node (L2) [scale=0.3] at (1.5,-1.73) {};
		\end{scope}
		
		\begin{scope}[every edge/.style={draw=black,thin}]
			\draw (G) edge (H);
			\draw (H) edge (I);
			\draw (I) edge (J);
			\draw (J) edge (K);
			\draw (K) edge (L);
			\draw (L) edge (G);
			
			\draw (G1) edge (H1);
			\draw (H1) edge (I1);
			\draw (I1) edge (J1);
			\draw (J1) edge (K1);
			\draw (K1) edge (L1);
			\draw (L1) edge (G1);
			
			\draw (G2) edge (H2);
			\draw (H2) edge (I2);
			\draw (I2) edge (J2);
			\draw (J2) edge (K2);
			\draw (K2) edge (L2);
			\draw (L2) edge (G2);
			
			\draw (G) edge (I);
			\draw (G) edge (J);
			\draw (G) edge (K);
			\draw (H) edge (J);
			\draw (H) edge (K);
			\draw (H) edge (L);
			\draw (I) edge (K);
			\draw (I) edge (L);
			\draw (J) edge (L);
			
			\draw (G1) edge (I1);
			\draw (G1) edge (J1);
			\draw (G1) edge (K1);
			\draw (H1) edge (J1);
			\draw (H1) edge (K1);
			\draw (H1) edge (L1);
			\draw (I1) edge (K1);
			\draw (I1) edge (L1);
			\draw (J1) edge (L1);
			
			\draw (G2) edge (I2);
			\draw (G2) edge (J2);
			\draw (G2) edge (K2);
			\draw (H2) edge (J2);
			\draw (H2) edge (K2);
			\draw (H2) edge (L2);
			\draw (I2) edge (K2);
			\draw (I2) edge (L2);
			\draw (J2) edge (L2);
		\end{scope}
		
		\begin{scope}[every edge/.style={draw=black,dashed}]
			\draw (A) edge (B);
			\draw (B) edge (C);
			\draw (C) edge (D);
			\draw (D) edge (E);
			\draw (E) edge (F);
			\draw (F) edge (A);
			
			\draw (A1) edge (B1);
			\draw (B1) edge (C1);
			\draw (C1) edge (D1);
			\draw (D1) edge (E1);
			\draw (E1) edge (F1);
			\draw (F1) edge (A1);
			
			\draw (A2) edge (B2);
			\draw (B2) edge (C2);
			\draw (C2) edge (D2);
			\draw (D2) edge (E2);
			\draw (E2) edge (F2);
			\draw (F2) edge (A2);
			
			\draw (A) edge (G);
			\draw (B) edge (H);
			\draw (C) edge (I);
			\draw (D) edge (J);
			\draw (E) edge (K);
			\draw (F) edge (L);
			
			\draw (A1) edge (G1);
			\draw (B1) edge (H1);
			\draw (C1) edge (I1);
			\draw (D1) edge (J1);
			\draw (E1) edge (K1);
			\draw (F1) edge (L1);
			
			\draw (A2) edge (G2);
			\draw (B2) edge (H2);
			\draw (C2) edge (I2);
			\draw (D2) edge (J2);
			\draw (E2) edge (K2);
			\draw (F2) edge (L2);
		\end{scope}
		
		\begin{scope}[every edge/.style={draw=black,thick,dotted}]
			\draw (C) edge (F1);
			\draw (B) edge (E2);
			\draw (D2) edge (A1);
		\end{scope}
	\end{tikzpicture}
	\caption{ A graphic matroid with $36$ vertices and $84$ edges where edges are styled according to the universal density $\eta^*$: $45$ solid edges are present with $\eta^* = \frac{1}{3}$, $36$ dashed edges with $\eta^* = \frac{1}{2}$ and $3$ dotted edges with  $\eta^* = \frac{2}{3}$.}\label{figure1}
	
\end{figure}
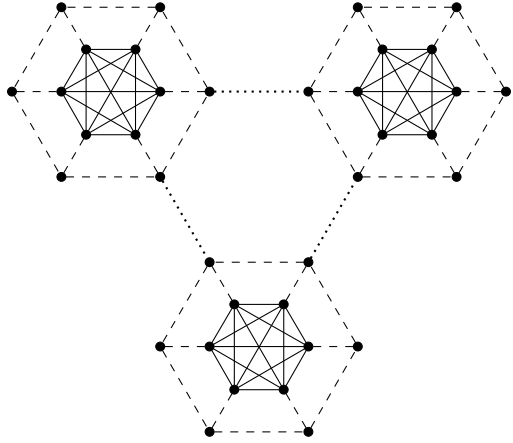

\begin{table}[t]
	\centering
	\caption{Values for the universal bases $\eta_t^*$ of matroid truncation $M_t$ for $t=1,\dots,84$, of the graphic matroid presented in Figure \ref{figure1}.}\label{table1}
	\resizebox{\textwidth}{!}{%
		\begin{tabular}{|c|c|c|c|c|c|c|}
			\hline
			Styles (\#)
			& $t = 1\to28$ 
			& $t = 29\to34$ 
			& $t = 35$ 
			& $t = 36\to42$ 
			& $t = 43\to55$ 
			& $t = 56\to84$ \\
			\hline
			Solid (45) 
			& $t/84$ 
			& $1/3$ 
			& $1/3$ 
			& $(t - 20)/45$ 
			& $(t - 2)/81$ 
			& $t/84$ \\
			\hline
			Dashed (36)
			& $t/84$
			& $(t - 15)/35$ 
			& $1/2$ 
			& $1/2$ 
			& $(t - 2)/81$
			& $t/84$ \\
			\hline
			Dotted (3)
			& $t/84$ 
			& $(t - 15)/35$ 
			& $2/3$ 
			& $2/3$ 
			& $2/3$ 
			& $t/84$ \\
			\hline
		\end{tabular}
	}
\end{table}

The next aspect of the modulus framework arises when the ground set $E$ is endowed with a set of weights. Let $M$ be a matroid with ground set $E$. Let $\si\in \R^E_{>0}$ be a set of weights attached to elements in $E$, we say that $M$ is weighted by $\si$. The universal density $\eta^*_M$ for the weighted matroid $M$ with weights $\si$ is defined as the unique optimal solution of the problem:
 \begin{equation}\label{eq:2norm-w}
 	\min \left\{ \sum\limits_{e\in E}\si^{-1}(e)\eta^2(e): \eta \in \co(\cB(M)). \right\}
 \end{equation}
  In this paper, we introduce a new characterization of the universal density for  weighted matroids. This new characterization arises from the following optimization problem motivated by the Kullback--Leibler divergence:
 \begin{equation}\label{eq:prodnorm}
 \underset{\eta \in \co(\cB(M))}{\text{minimize}}\quad-\sum_{e\in E}\si(e)\log\eta(e),
 \end{equation}
where we define $\log\eta(e) = -\infty$ when  $\eta(e) =0$, for $e\in E$. Specifically, we will show that the universal density of $M$ is the unique optimal solution for the problem (\ref{eq:prodnorm}), see Theorem \ref{thm:mainmkl}.
\vspace{\baselineskip}

The next topic we focus on is the notion of homogeneous matroids. A matroid $M$ with weights $\si$ is said to be {\it $\si$-homogeneous} if $\si^{-1}(e)\eta^*_M(e)$ is a constant for all $e \in E$. When $M$ is unweighted, such a matroid $M$ is also referred to as a {\it uniformly dense matroid} \cite{catlin1992,uniform-karel}. In \cite{catlin1992,truong2024modulus}, the authors present several characterizations of homogeneous matroids. One such characterization is that $M$ is homogeneous if and only if $\theta(M) = D(M)$. Independently, in \cite{ordering88}, uniformly dense matroids appeared in a conjecture on orderings of the bases of a matroid. The conjecture says that a matroid $M$ is uniformly dense if and only if there exists an ordering of $E$ such that all $r(M)$ cyclically consecutive elements form a base of $M$. Van den Heuvel
and Thomass\'e \cite{thomasse}
proved this conjecture for
matroids whose size and rank are coprime. In the graph case, the authors in \cite{albin2024minimizing} introduce the notion of {\it strictly homogeneous} graphs. One of its characterizations is that a biconnected graph is strictly homogeneous if it is homogeneous and it contains no strict subgraph with the same graph density.
In Section \ref{strictly-hom} of this paper, we generalize this notion to the case of matroids and also give several equivalent characterizations.\\

The next contribution of this paper is to extend the study of the base modulus from \cite{truong2024modulus}, providing several related results. In Section \ref{sec:sd}, we give some properties related to the strength and the fractional arboricity in weighted matroids. In Section \ref{sec:dualmatroid}, we extend a result in \cite{truong2024modulus} of the fundamental relationship between the universal base of a matroid and its dual to the weighted case. Next, the set of pmfs $\mu \in \R^{\cB}_{\geq 0}$ whose marginals equal to $\eta^*_M$, i.e.,
\[\sum\limits_{B\in \cB} \mu(B)\ones_B = \eta^*,\]where $\ones_B$ is the indicator function of a base $B$, plays an important role in the study of base modulus and strictly homogeneous matroids. The problem of finding all such pmfs has not been fully investigated. 
In Section \ref{sec:pmf}, we describe such pmfs in a simple case where $M$ is the union of two matroids of rank 1. This provides a starting point for investigating the set of all such pmfs in the general case.

\vspace{\baselineskip}

Finally, this paper also addresses optimization problems for graph structures, particularly those involving edge-disjoint spanning trees.
For a graph $G$, the problem of determining the minimum number of edges to be added so that the resulting graph has $k$ edge-disjoint spanning trees was explored by Haas \cite{haas2002} and Liu et al. \cite{liu2009reinforcing}. The analogous problem for matroids was studied in \cite{lai2010reinforcing}, in which the minimum number of elements to be added was expressed in terms of certain invariants of the matroid. For matroids with $k$ disjoint bases, the set $E_k(M)$ of elements such that the removal of any $e \in E_k(M)$ leaves the matroid with $k$ disjoint bases has been characterized using the matroid strength by Li et al. \cite{li2012characterization}.  
Recall that, for a matroid $M$, $a(M)$ is the minimum number of bases whose union equals $E(M)$. In analogy, for a graph $G$, the arboricity $a(G)$ is the minimum number of spanning trees whose union equals $E(G)$. Despite the advances mentioned above, the following two related problems remain less explored: Given a matroid $M$ with $a(M) > h$, how many elements should be deleted from $M$ so that the resulting matroid $M'$ has $a(M') \leq h$? Given a graph $G$ with $a(G) \leq h$, what kind of new edges $e$ can be added to $G$ so that $a(G+e) \leq
h$? We answer the first question in Section \ref{sec:removing} and answer the second question in Section \ref{addable}.

\vspace{\baselineskip}

Here is a summary of our results in this paper:
\bi

\item In Section \ref{sec:truncation}, we provide the universal densities for all matroid truncations.
\item  In Section  \ref{sec:kl}, we provide a new characterization of the universal density of weighted matroids, arising from the Kullback--Leibler divergence.

\item In Section \ref{strictly-hom}, we introduce strictly $\sigma$-homogeneous matroids and give an equivalent characterization.

\item In Section \ref{sec:sd}, we provide several monotonicity and comparison lemmas related to $S_{\si}(M)$ and $D_{\si}(M)$ for a weighted matroid $M$ with weights $\si \in \R^E_{\geq0}$.

\item In Section \ref{sec:dualmatroid}, we extend the relationship between the universal density of a matroid and its dual.

\item In Section \ref{sec:pmf}, we describe all pmfs that induce the universal density in a simple case where $M$ is the union of two matroids of rank $1$. 

\item In Section \ref{sec:removing}, we answer the question: Given a matroid $M$ with $a(M) > h$, how many elements should be deleted from $M$ so that the resulting matroid $M'$ has $a(M') \leq h$? 

\item In Section \ref{addable}, we answer the question:  Given a graph $G$ with $a(M) \leq h$, what kind of new edges $e$ can be added to $G$ so that $a(G+e) \leq
h$?
\ei

\section{Preliminaries: Modulus for bases of matroids}\label{sec:mod-base}
For background on matroids and the discrete modulus, we refer the reader to Appendix~\ref{subsec:mat} and Appendix~\ref{subsec:mod}, respectively.

\vspace{\baselineskip}

Let $M=(\cI,E)$ be a matroid with given weights 
$\sigma \in \mathbb{R}^E_{>0}$ assigned to
elements in $E$. Let $\cB=\cB(M)$ be the family of bases of $M$ where the usage vectors are indicator functions. Let $\mathcal{N}$ be the matrix whose rows are the indicator vectors of the bases. In \cite{reinforcement}, the authors provided the weighted version for modulus of the base family of matroids.  We now recall some notations.
Given weights $\si$, for any subset $X \subseteq E(M)$, denote $\si(X)=\sum\limits_{e\in X}\si(e).$ For any subset $X \subseteq E(M)$ with $r(X) > 0$, the {\it weighted density} of $X$ is defined as
$\theta_{\si}(X) = \si(X)/r(X).$ We also use $\theta_{\si}(M)$ for $\theta_{\si}(E)$. Next, we define the {\it weighted strength} $S_{\si}(M)$ of $M$ as
\begin{equation}\label{eq:weighted-strength-problem}
	S_{\si}(M) := \min \left\{ \frac{\si(X)}{r(E) - r(E-X)} : X \subseteq E, r(E) > r(E-X) \right\},
\end{equation}
and the {\it weighted fractional arboricity} $D_{\si}(M)$ of $M$:
\begin{equation}\label{eq:weighted-d-problem}
	D_\si(M) := \max \left\{ \frac{\si(X)}{r(X)} : X \subseteq E, r(X)>0 \right\}.
\end{equation}

Let $\eta^*:=\eta^*_M$ be the universal density of $M$ with weights $\si$. We denote $\si^{-1}\eta^*$ the vector of $\{\si^{-1}(e)\eta^*(e): \ e \in E \}.$
Let $(\si^{-1}\eta^*)_{min}:=s_1<s_2<\dots<s_k=: (\si^{-1}\eta^*)_{max}$ be the distinct values of $\si^{-1}\eta^*$. For $i=1,\dots,k$, denote 
\begin{equation}\label{eq:Ai}
	A_i := \{e \in E : \si^{-1}(e)\eta^*(e) = s_i\}
\end{equation} and 
\begin{equation}\label{eq:Ei}
	E_i:= \{e \in E: \si^{-1}(e)\eta^*(e)\geq s_i\}.
\end{equation}
 Then, we have $E_1 = E$, and for 
$j = 2, \dots, k$, 
\[
E_i = E - \bigcup\limits_{j=1}^{i-1} A_j.
\]
In \cite{truong2024modulus}, it was shown that the matroid contraction $M/(E-E_i)$ satisfies:
\begin{equation}\label{eq:pre}
	S_{\si}(M/(E-E_i))= \frac{1}{s_k}, \quad D_{\si}(M/(E-E_i))= \frac{1}{s_i}.
\end{equation}
 Also, $E_k$ is the unique maximal optimal set for  $S_{\si}(M/(E-E_i))$, and $A_i$ is the unique maximal optimal set for $D_{\si}(M/(E-E_i))$. Moreover, the universal density of $ M/(E-E_i)$ is the restriction of $\eta^*$ to $E_i$. In particular, $E_k$ is the unique maximal optimal set for  $S_{\si}(M)$, and $A_1$ is the unique maximal optimal set for $D_{\si}(M)$. From now on, we will call $A_1$ the {\it core} of $M$. Note that the sets $A_i$ are the blocks of the principal partition of
 the matroid; see \cite{fujishige2009theory} for further details on the
 principal partition.

\begin{remark}\label{re:def-hom}
	We recall that a matroid $M$ is $\si$-homogeneous if the vector $\si^{-1}\eta^*_M$ is a constant vector. It was shown in \cite{truong2024modulus} that $M$ is $\si$-homogeneous if and only if there is a pmf $\mu \in \cP(\cB)$ such that $\cN^T\mu$ is parallel with $\si$. It was also shown that 
	\begin{align}\label{eq:hom-char}
		M \text{ is $\si$-homogeneous} \Leftrightarrow S_\si(M)=\theta_\si(M) \Leftrightarrow \theta_\si(M)=D_\si(M) \Leftrightarrow S_\si(M)=D_\si(M).
	\end{align}
\end{remark}  

\begin{remark}\label{re:modinf}
	It was also shown in \cite{truong2024modulus} that $S_\si(M) = \Mod_{1,\si}(\cB)$.
	 Combining this with the result $S_\si(M) = 1/ s_k$ from (\ref{eq:pre}) with $ i=1$, we obtain that
	\begin{equation}\label{eq:inf}
		 s_k=\frac{1}{S_\si(M)} = \frac{1}{\Mod_{1,\si}(\cB)}=\Mod_{\infty,\si^{-1}}(\widehat{\cB})= \min\{ \max\limits_{e\in E} \{\si^{-1} (e)\eta(e)\} :\eta \in \co(\cB)\},
	\end{equation}
	where the third equality follows from Fulkerson duality for modulus in (\ref{eq:dual-infty}) and the last equality follows from (\ref{eq:adm-dom}).
 In particular, $\eta^*_M$ is also an optimal solution for the minimization problem in (\ref{eq:inf}).
\end{remark}

\vspace{\baselineskip}

Next, we generalize the notion of strictly homogeneous graphs to the context of matroids and provide some results from \cite{albin2024minimizing} in this context.
With given weights $\la\in\R_{>0}^E$ on $E$. Let $\mu_{\la}$ be a pmf in $\cP(\cB)$ such that, for every base $B\in \cB$,  the probability density $\mu_{\la}(B)$ is proportional to
\begin{equation}\label{eq:mu-la}
	\la[B] := \prod\limits_{e\in B}\la(e).
\end{equation}  We denote $\eta_{\la}$ the element usage induced by $\mu_\la$, meaning that
\begin{equation}\label{eq:peredgeeffres}
	\eta_{\la}(e) := \cN^T{\mu_\la}.
\end{equation}
Then, we introduce the notion of {\it strictly homogeneous matroids} as follows. 
\begin{definition}\label{def:s-h}
	Let $M=(E,\cI) $ be a matroid with given weights $\si \in\R^E_{>0}$. A matroid $M$ is said to be {\it strictly $\si$-homogeneous} if there exists weights $\la\in\R_{>0}^E$ such that $\eta_{\la}$ is parallel with $\si$, in other words,
	\[\eta_{\la}(e) = k\si(e), \text{ for some } k>0.\]
\end{definition}
To state some results from \cite{albin2024minimizing}, we need several notations. For any vector $\beta \in \R^E_{\geq 0}$, the set of $\beta$-induced pmf is denoted by  $\cU_\beta \subset \cP(\cB)$:
\begin{equation} \label{eq:cone_definition}
	\cU_\beta := \left\{ \mu \in \cP(\cB) : \cN^T\mu = \beta	 \right\}.
\end{equation}
Assuming the weights $\beta$ are chosen so that $\mathcal{U}_\beta$ is nontrivial, i.e, $\beta \in \co(\cB)$. We denote $\cB_{\beta}$ the set of all bases $B$ that belong to the support set of some pmf $\mu \in \cU_\beta $:
\begin{equation}
	\cB_{\beta} := \bigcup_{\mu\in \mathcal{U}_\beta}\supp \mu.
\end{equation}
The following problem of minimizing the Shannon entropy over $\mathcal{U}_\beta$ was studied in \cite{albin2024minimizing}:
\begin{equation}\label{eq:max-entropy}
	\underset{\mu\in\mathcal{U}_\beta}{\text{maximize}}\quad -\sum_{\gamma\in\Gamma}\mu(\gamma)\log\mu(\gamma) =: H(\mu),
\end{equation}
(here we interpret $ \mu(\gamma) \log \mu(\gamma) = 0 $ if $ \mu(\gamma) = 0 $).
It was shown in \cite[Lemma 3.8]{albin2024minimizing} that the optimal pmf $\mu^*$ of problem (\ref{eq:max-entropy}) is unique.  We call $\mu^*$ the unique maximum-entropy pmf in $\mathcal{U}_\beta$.

\vspace{\baselineskip}

Another optimization problem from \cite{albin2024minimizing} is described as follows.
Let $ M = (E, \cI) $ be a matroid with given weights $ \sigma \in \mathbb{R}^E_{>0} $.
Following \cite{albin2024minimizing}, we consider the problem:
\begin{equation}\label{eq:min-det}
	\begin{split}
		\underset{\la>0}{\text{minimize}}\quad&\sum_{B\in\cB}\la[B]\\
		\text{subject to}\quad&\prod_{e\in E}\la(e)^{\si(e)} = 1.
	\end{split}
\end{equation}
Given weights $\si \in \R_{>0}^{E}$, denote a dilation $	\beta_{\si}$ of $\si$ as follows:
\begin{equation}\label{eq:beta-si}
	\beta_{\si}(e) := \frac{\si(e)}{\si(E)}r(E) \quad\text{ for } e \in E.
\end{equation}
We generalize a result from \cite{albin2024minimizing} that gives a relation between the problem \ref{eq:max-entropy} and the problem \ref{eq:min-det}.

\begin{theorem}
	\label{thm:main-mindet}
	Let $M=(E,\cI) $ be a matroid with given weights $\si \in\R^E_{>0}$. Let $\beta_{\si}$ be defined as in (\ref{eq:beta-si}).
	If $\mathcal{U}_{\beta_{\si}}$ is trivial, then the infimum in~\eqref{eq:min-det} is zero.  Otherwise, this infimum is equal to $\exp(H(\mu^*))$, where $\mu^*$ is the unique maximum-entropy pmf in $\mathcal{U}_{\beta_{\si}}$.  Let $\{\la_k\}$ be a minimizing sequence of~\eqref{eq:min-det}, and let $\mu_{\la_k}$ be the corresponding pmfs.  Then $\mu_{\la_k}\to\mu^*$. Moreover, a set of weights $\la\in\mathbb{R}^E_{>0}$ minimizes the problem~\eqref{eq:min-det} if and only if $\mu_{\la}=\mu^*$.
\end{theorem}
The proof of Theorem \ref{thm:main-mindet} is the same as the proof in the case of spanning trees provided in \cite{albin2024minimizing}.
\begin{remark}\label{rem:min-strict}
	If the problem~\eqref{eq:min-det} has a minimizer, by Theorem \ref{thm:main-mindet}, there exists a set of weights $\la\in\mathbb{R}^E_{>0}$ such that $\mu_{\la}=\mu^*$, where $\mu^*$ is the unique maximum-entropy pmf in $\mathcal{U}_{\beta_{\si}}$. It follows that $\eta_{\la} = \beta_{\si}$, in other words, $M$ is strictly $\si$-homogeneous.
\end{remark}
In \cite{albin2024minimizing} , the authors gave a characterization for strictly homogeneous graphs, this result is stated in the context of matroids as follows. 
\begin{theorem}
	\label{thm:supp}
	Let $M=(E,\cI) $ be a matroid. Let $\cB$ be the base family of $M$ and let $\beta \in \co(\cB)$. There exists a set of weights $\la\in\R_{>0}^E$ such that $\eta_{\la}=\beta$ if and only if $\cB_{\beta}=\cB$. The set of such $\beta$ is denoted by $\co^+(\cB)$.
\end{theorem}

\section{Truncation of matroids and modulus}\label{sec:truncation}

\subsection{Homogeneous matroids and balancity}
Let $M=(E,\cI)$ be a matroid with the base family $\cB$.
Let $\mathcal{N}$ be the usage matrix whose rows are the indicator vectors of the bases.
First, we recall the notion of \emph{matroid truncation}. Let $t$ denote a fixed
integer such that $1 \leq t \leq r(M)$. The \emph{$t$-truncation} of $M$,
denoted by $M_t$, is a matroid on $E$ whose collection of independent sets is
the set of all independent sets of $M$ with size at most $t$.
For $X \subset E$, the rank of $M_t$ is
\[
r_{M_t}(X)=\min\{t, r_M(X)\},
\]
and the density of $M_t$ is
\begin{equation}\label{eq:thetat}
	\theta(M_t)=\frac{|E|}{t}=:\theta_t(M).
\end{equation}

Moreover, the strength and fractional arboricity of $M_t$ satisfy:
\begin{equation}\label{eq:st}
	S(M_t)= \min\left\{ \frac{|X|}{t - r_M(E-X)}:X\subseteq E,\ r_M(E-X)<t\right\}=:S_t(M)
\end{equation}
and
\begin{equation}\label{eq:dt}
	D(M_t)=\max\left\{ \frac{|X|}{\min\{t, r_M(X)\}}:\emptyset \neq X \subseteq E\right\}
	=:D_t(M).
\end{equation}

The main goal of this section is to show that the universal density $\eta_t^*:=\eta_{M_t}^*$ of $M_t$ can be computed using the universal density $\eta^*:=\eta^*_M$ of $M$.  Recall that the matroid $M$ is homogeneous if and only if the universal density $\eta^*$ is constant. Since $\eta^*(E)=r(M)$, then $\eta^*(e)=r(M)/|E|$ for all $e\in E$ in this case.
We start with the following lemma.
\begin{lemma}\label{lem:mt-sh}
	If $M$ is (strictly) homogeneous, then $M_t$ is (strictly) homogeneous for any $t=1,\dots,r(M)-1.$
\end{lemma}
\begin{proof}
	Assume that $M$ is homogeneous. Let $\mu^* \in \cP(\Ga)$ be an $\eta^*$-induced pmf, i.e., $\cN^T\mu^*=\eta^*$. Denote $r:=r(M)$, then \[\eta^*(e)=\sum\limits_{B\in \cB:e \in B}\mu^*(B)=r/|E|, \quad \forall e\in E.\]
Let $\ga$ be an independent set of size $t$. We define $\mu$ 
 as follows: \[\mu(\ga):= \sum\limits_{B\in \cB:\ga\subset B}\mu^*(B)\frac{1}{\binom{r}{t}}.\] We want to show that  $\mu $ is a pmf in $ \cP(M_t)$ with constant element-usages. Indeed, $\mu$ is a pmf because any base $B$ contains exactly $\binom{r}{t}$ subsets of size $t$ which are independent. Moreover, if we fix $e\in E$:

\begin{align*}
	\sum\limits_{\ga\in \cB(M_t):e \in \ga}\mu(\ga)&=	\sum\limits_{\ga\in \cB(M_t):e \in \ga}\sum\limits_{B\in \cB(M):\ga\subset B}\mu^*(B)\frac{1}{\binom{r}{t}}\\
	&=\frac{1}{\binom{r}{t}}	\sum\limits_{\ga\in \cB(M_t)}\sum\limits_{B\in \cB(M)}\ones_{e\in \ga}\ones_{\ga\subset B}\mu^*(B)\\
	&=\frac{1}{\binom{r}{t}}\sum\limits_{B\in \cB(M)}\binom{r-1}{t-1}\ones_{e\in B}\mu^*(B)\\
	&=\frac{1}{\binom{r}{t}}\binom{r-1}{t-1}\eta^*(e)\\
	&=\frac{t}{r}\frac{r}{|E|}=\frac{t}{|E|}.
\end{align*}

Hence, $M_t$ is homogeneous. Moreover, if $\mu^*(B)>0$ for all $B\in \cB$, then $\mu(\ga)>0$ for all $\ga\in M_t$ and $M_t$ is strictly homogeneous.
\end{proof}

The following lemma provides several properties of the sequence of truncations $M_t$.
\begin{lemma}\label{lem:sequence-sd}
	Let $M=(E,\cI)$ be a matroid. Let $t$ be an integer such that $1\leq t<r(M)$ and let $M_t$ be the $t$-truncation of $M$. Then, $D(M_t)\geq D(M_{t+1})$, $\theta(M_t)> \theta(M_{t+1})$, and $S(M_t)> S(M_{t+1})$.
\end{lemma}

\begin{proof}
	First, we have 
	\[ D(M_t) =\max_{\emptyset \neq X \subseteq E} \frac{|X|}{\min\{t, r_M(X)\}} \geq  \max_{\emptyset \neq X \subseteq E} \frac{|X|}{\min\{t+1, r_M(X)\}} = D(M_{t+1}).\]
	Next, we have \[\theta(M_t) = \frac{|E|}{t}> \frac{|E|}{t+1}=\theta(M_{t+1}).\]
	Finally, let $A$ be an optimal set for $S(M_t)$. Then
	\[ S(M_t)= \frac{|E| - |A|}{t - r_M(A)} >\frac{|E| - |A|}{t+1 - r_M(A)} \geq S(M_{t+1}).\]
\end{proof}

 We note that $M_1$ is always homogeneous. By Lemma \ref{lem:mt-sh}, it follows that if $M_{t_0}$ is homogeneous for some $1<t_0 \leq r(M)$, then $M_{t}$ is also homogeneous for $t <t_0$. Hence, we introduce the following notion.
 \begin{definition}
 		Let $M=(E,\cI)$ be a matroid. The largest integer $t$ such that $M_t$ is homogeneous is called the {\it balancity} of $M$ and is denoted by $b(M)$.
 \end{definition} This definition is a generalization of the balancity of graphs in \cite{the-higher-order}. The following proposition provides a characterization of the balancity $b(M)$ of $M$. Recall that, for a matroid $M$, the unique maximal optimal set for $D(M)$ is called the core of $M$.

\begin{proposition}\label{mt-hom}
	Let $M=(E,\cI)$ be a matroid. Let $t$ be an integer such that $1\leq t \leq r(M)$. Then,	$M_t$ is homogeneous if and only if $D(M) \leq \theta_t(M)$. So,
	\[b(M)= \max\left\{ t\in \Z\cap[1,r(M)] :D(M) \leq \theta_t(M)\right\}.\]
	If $t > b(M)$, then $M_t$ is not homogeneous, and $D(M)=D(M_{t})$, furthermore, $M$ and $M_t$ share the same core $X^*$ with $ r(X^*)<t$.
\end{proposition}

\begin{proof}
	For the case $t=r(M)$, the statement holds because of (\ref{eq:hom-char}) with constant $\si$. Next, for an integer $t$ such that $ t\leq r(M)-1$, we have
	\begin{equation}\label{eq:break}
		D(M) = \max\left\{\max\limits_{0<r(X)< t}\frac{|X|}{r(X)}, \max\limits_{r(X)\geq t}\frac{|X|}{r(X)}\right\}.
	\end{equation}
	By (\ref{eq:dt}), we also have
	\begin{equation}\label{eq:dtm}
		D_t(M) = \max\left\{\max\limits_{0<r(X)< t}\frac{|X|}{r(X)}, \max\limits_{r(X)\geq t}\frac{|X|}{t}\right\}=\max\left\{\max\limits_{0<r(X)< t}\frac{|X|}{r(X)}, \frac{|E|}{t}\right\}.
	\end{equation}
	Note that we always have
	\begin{equation}\label{eq:obs}
		\max\limits_{r(X)\geq t}\frac{|X|}{r(X)} \leq \frac{|E|}{t}.
	\end{equation}
	Therefore,  $D(M)\leq \frac{|E|}{t}$, if and only if, $\max\limits_{0<r(X)< t}\frac{|X|}{r(X)}\leq \frac{|E|}{t}$, and if and only if, $D_t(M) = \frac{|E|}{t}=\theta_t(M)$. Hence, $D(M)\leq \frac{|E|}{t}=\theta_t(M)$, if and only if, $M_t$ is homogeneous.

	Assume that $M_t$ is not homogeneous, then  $D_t(M)> \frac{|E|}{t}$. By (\ref{eq:dtm}) and \eqref{eq:obs}, we obtain 
	
	\[D_t(M) =\max\limits_{0<r(X)< t}\frac{|X|}{r(X)}> \frac{|E|}{t} \geq	\max\limits_{r(X)\geq t}\frac{|X|}{r(X)}.\]
	Combining this with (\ref{eq:break}), we have
		\[D(M)=\max\limits_{0<r(X)< t}\frac{|X|}{r(X)}=D_t(M)> \frac{|E|}{t} \geq	\max\limits_{r(X)\geq t}\frac{|X|}{r(X)}.\]
	Therefore, the sets of solutions of $D(M)$ and $D(M_{t})$ are identical to the set of solutions of  $\max\limits_{0<r(X)< t}\frac{|X|}{r(X)}$. 
	Note that, as mentioned after (\ref{eq:pre}), the fractional arboricity problems $D(M)$ and $D(M_{t})$ admit unique maximal cores $X^*$ and $X_t^*$, respectively. It follows that $X^* = X_t^*$ and $ r(X^*)<t$.
\end{proof}

\subsection{Universal densities for all truncations}

The main goal of this section is to find the universal density $\eta^*_t$ of $M_t$ for all $t \leq r(M)$. For $t\leq b(M)$, the truncation $M_t$ is homogeneous by definition of $b(M)$, and in this case, the universal density of $M_t$ is the constant $t/|E|$. Therefore, we shall focus on the range when $t=b+1,\dots,r(M)$. In Section \ref{sec:dual-trun}, we will address the case $ t>r(M)$.

\begin{theorem}\label{thm:main-truncation}
	Let $M$ be a matroid with  rank function $r$. Let $s_1<s_2<...<s_k$ be the distinct values of the universal density $\eta^*$ of $M$. For $i=1,\dots,k$, denote $E_i:= \{e \in E: \eta^*(e)\geq s_i\}$, 
	$b_i:=\lfloor |E_i| s_i\rfloor$, and $c_i:=\lfloor |E_i| s_i\rfloor+r(E-E_i)$, where $\lfloor \cdot \rfloor$ is the floor function. Denote $c_0:= 0$.
	
	 For $t=1,\dots,r(M)$, let $\eta_t^*$ be the universal density of the $t$-truncation $M_t$ of $M$. Then, we have the following:
	\bi
	\item[(i)] $b_i$ is the balancity of $M/(E-E_i)$ for $i=1,\dots,r(M)$;
	\item[(ii)]  $0=c_0<c_1\leq c_2\leq\dots\leq c_{k-1}<c_k=r(M);$
	\item[(iii)]$c_i-b_i=r(E-E_i)\leq c_{i-1}$ for $i=1,\dots,r(M)$; 
	\item[(iv)] For $i=1,\dots,k$, and $t$ such that $c_{i-1}<t\leq c_{i}$ , then  
	
	\begin{equation}\label{eq:etat}
		\eta^*_{t}(e)=
		\begin{cases}
			\displaystyle\frac{t-r(E-E_i)}{|E_i|} & \text{if } e \in E_i; \\
			\eta^*(e)   & \text{if } e \in E-E_i.
		\end{cases}
	\end{equation}
	\ei
\end{theorem}

In order to prove Theorem \ref{thm:main-truncation}, we first want to study the relation between  contractions using the core  and truncations of a matroid.
When $b(M)<r(M)$, the matroids $M_{b+1},\dots, M_{r(M)}$ are not homogeneous. By Proposition \ref{mt-hom}, these matroids share the same core $X^*$ with $r(X^*)\leq b$. The following lemma shows that the matroid contraction $(M_t)/X^*$ is identical to the $(t-r(X^*))$-truncation $(M/X^*)_{t-r(X^*)} $  of $M/X^*$. With this lemma, we will be able to study the contraction $(M_t)/X^*$ by applying Proposition \ref{mt-hom} to matroid truncations of $M/X^*$. 
 
\begin{lemma}\label{lem:tr-tr}
	Let $b=b(M)$ be the balancity of $M$. Assume that $b<r(M)$. Let $X^*$ be the core of $M$. For $t=b+1,\dots,r(M)$,	we have that $(M_t)/X^*$ is identical to the $(t-r(X^*))$-truncation  of $M/X^*$.
\end{lemma}
\begin{proof}
	First, by Proposition \ref{mt-hom}, all matroid truncations $M_{b+1},\dots, M_{r(M)}$ share the same core $X^*$ where $r(X^*) \leq b$, and $X^*$ is the unique maximal optimal set for
	\[\max\limits_{0<r(X)\leq b}\frac{|X|}{r(X)}.\]
	 It follows from $r(X^*) \leq b$ that $M|X^*=M_t|X^*$ for $t = b+1,\dots,r(M)$. 
	For any $t \geq b+1$, by the definition of contraction, a subset $I \in E$ is independent in $(M_t)/X^*$ if and only if $\exists B \in \cB(M_t|X^*)$ such that $B \cup I $ is independent in $M$ and $r(X^* \cup I) \leq t$.
	Conversely, a subset $I \in E$ is independent in $(M/X^*)_{t-r(X^*)} $ if and only if $\exists B \in \cB(M|X^*)$ such that $B \cup I $ is independent in $M$ and $r(X^* \cup I)-r(X^*) \leq t-r(X^*)$, which is equivalent to $r(X^* \cup I) \leq t$.
	Therefore, $(M_t)/X^*$ and $(M/X^*)_{t-r(X^*)}$ share the same collection of independent sets.
\end{proof}
Next, we want to study the balancity of $M/X^*$. If $r(X^*)=b(M)$, then $(M/X^*)_{b(M)-r(X^*)}$ has rank $0$.  If $r(X^*)<b(M)$, the following lemma shows that the balancity of $M/X^*$ is at least $b(M)-r(X^*)$. 
\begin{lemma}\label{lem:shrink-hom}
	Let $b=b(M)$ be the balancity of a matroid $M$. Let $X^*$ be the core of $M$. If $r(X^*)<b(M)$, then  $(M/X^*)_{b-r(X^*)}$  is homogeneous. And in fact, it is strictly homogeneous.
\end{lemma}

\begin{proof}
	Denote $t= b-r(X^*)$ and $\cM:=M/X^*$. We want to show that $D(\cM_t) = \theta(\cM_t)$. Recall that $\cM$ has ground set $E-X^*$, and
for any nonempty subset $A\subset E-X^*$, we have $r_{\cM}(A)=r(X^*\cup A)-r(X^*)$. By \ref{eq:dtm}, we have

\begin{equation*}
	D(\cM_t) =\max\left\{\max\limits_{0<r_{\cM}(A)< t}\frac{|A|}{r_{\cM}(A)}, \frac{|E-X^*|}{t}\right\}.
\end{equation*}
	We want to show that 
	\begin{equation}\label{eq:1<2}
		\max\limits_{0<r_{\cM}(A)< t}\frac{|A|}{r_{\cM}(A)}< \frac{|E-X^*|}{t}=\theta(\cM_t).
	\end{equation}
	 Since $X^*$ is closed in $M$, it is enough to consider sets $A \neq \emptyset$ such that $r(X^*\cup A)-r(X^*) < t$. By the definition of $X^*$, we have
	\begin{equation}\label{eq:defX}
		\frac{|X^*\cup A|}{r(X^*\cup A)}<\frac{|X^*|}{r(X^*)}.
	\end{equation}
	Since  $r(X^*)<b(M)$, we have $r(X^*)= r_{M_b}(X^*)$. And since $M_b$ is homogeneous, we have
	\begin{equation}\label{eq:Mbhom}
		\frac{|X^*|}{r(X^*)}= \frac{|X^*|}{r_{M_b}(X^*)} \leq \frac{|E|}{b}.
	\end{equation}
	Hence,
	\begin{align*}
	\frac{|A|}{r_\cM(A)}	=	\frac{|A|}{r(X^*\cup A)-r(X^*)}&=	\frac{|X^*\cup A|-|X^*|}{r(X^*\cup A)-r(X^*)} 
			< \frac{|X^*|}{r(X^*)} & (\text{by (\ref{eq:defX})})\\
			&\leq \frac{|E|}{b} \leq \frac{|E|-|X^*|}{b-r(X^*)} &(\text{by (\ref{eq:Mbhom})}).
	\end{align*}
	Hence, $D(\cM_t) = \theta(\cM_t)$. It follows that
	$\cM_t$ is homogeneous. 
	In fact, it is strictly homogeneous because of the strict inequality in (\ref{eq:1<2}).
\end{proof}

Finally, we are ready to prove Theorem \ref{thm:main-truncation}. By Lemma \ref{lem:tr-tr}, every contraction $M_t/X^*$ is identical to a truncation of $M/X^*$. So, applying Proposition \ref{mt-hom} and Lemma \ref{lem:tr-tr} to $M/X^*$
will reveal additional values of $\eta_t^*$. Continuing by contracting the core of $M/X^*$ and applying these two lemmas to further reveal values of $\eta_t^*$. This strategy determines all the values of $\eta_t^*$.

\begin{proof}[Proof of Theorem \ref{thm:main-truncation}]
	Denote $A_i := \{e \in E : \eta^*(e) = s_i\}$. By definition, we have $E_1 = E$, and for 
	$j = 2, \dots, k$, 
	\[
	E_i = E - \bigcup\limits_{j=1}^{i-1} A_j.
	\]
We prove this theorem by induction as follows.
	By (\ref{eq:pre}), we have $s_1= 1/D(M)$ and $M$ has the core $A_1$. By Proposition \ref{mt-hom}, the balancity of $M$ is equal to 
	$\lfloor |E| / D(M) \rfloor=\lfloor |E_1| s_1\rfloor=b_1$. Thus, for $c_0=0<t \leq c_1 = b_1$, we have that $M_t$ is homogeneous, equivalently, $\eta_t^*$ is a constant vector, and
	\[\eta_t^*(e) =\frac{t}{E_1} =\frac{t - r(E - E_1)}{E_1}.\]
	So, (i),(ii),(iii), and (iv) hold for $i=1$.
	
		By Proposition \ref{mt-hom}, since $A_1$ is the core of $M$, it is also the core of $M_t$ for $t > c_1=b_1$, hence
	\[\eta_t^*(e) = \eta^*(e) = s_1 \text{ for } e \in A_1=E-E_2.\]
	Next, we consider the matroid $M / A_1 = M / (E - E_2)$. By Proposition \ref{mt-hom}, the 
	balancity of $M / (E - E_2)$ is equal to 	$\lfloor |E_2| / D(M/ (E - E_2)) \rfloor=\lfloor |E_2| s_2\rfloor=b_2$, and $M / (E - E_2)$ has the core $A_2$. By 
	Lemma \ref{lem:shrink-hom}, the balancity $b_2$ of 
	$M / (E - E_2)$ is at least $b_1 - r(E - E_2)$. Therefore, 
	\[c_2 = b_2 + r(E - E_2) \geq b_1 = c_1.\] By Proposition \ref{mt-hom}, 
	we also have \[r(E - E_2) \leq b_1 = c_1.\]
	Therefore, if $t$ satisfies $c_1 < t \leq c_2$, by Lemma \ref{lem:tr-tr}, 
	$M_t / (E - E_2) = (M / (E - E_2))_{t - r(E - E_2)}$ and this truncation is homogeneous 
	since $t - r(E - E_2) \leq b_2$, equivalently,
	\[
	\eta_t^*(e) = \frac{t - r(E - E_2)}{E_2}, \quad \text{for } e \in E_2.
	\]
	So, (i),(ii),(iii), and (iv) hold for $i=2$.
	Continue this process until the proof is completed.

\end{proof}
\begin{remark}
	With the assumptions in Theorem \ref{thm:main-truncation},
	if $c_{i-1}<t\leq c_{i}$ for $i>1$, then $\eta^*_t$ attains $i$ distinct values, and $\eta_t^*$ agrees with $\eta^*$ in its first $i-1$  distinct values in increasing order.
	When we decrease $t$ from
	$c_k=r(M)$ to $c_0=0$, the blocks $A_i$, which form a partition of $E$
	induced by $\eta^*$, start to merge together, in the sense that, the
	partition induced by $\eta_t^*$ is obtained by successively merging
	blocks of the principal partition as $t$ passes through the critical
	values $c_i$. Note that $c_{i-1}$ may equal $c_i$ for several $i$, so two or more
	blocks can merge at a time.
	
\end{remark}

\subsection{Dual truncations}\label{sec:dual-trun}
In Theorem \ref{thm:main-truncation},  we provided the universal density of all matroid truncations of a given matroid $M$. Next, we consider a different collection of matroids. Let $M^*$ be the dual matroid of $M$. It was shown in \cite{truong2024modulus} that the universal density of $M^*$ is equal to $\one - \eta^*$, where $\one$ is the vector of all ones.
For $t=r(M),r(M)+1,\dots,|E|-1$, we define the {\it dual truncation} $M_t$ as follows:

\begin{equation}\label{eq:addedge}
	M_t:= ((M^*)_{|E|-t})^*,
\end{equation}where $M^*$ denotes the dual matroid of $M$. For consistency, when $t=|E|$, we denote by $M_t$ the matroid whose ground set $E$
is a base. It follows that the base family of $M_t$ is
\begin{equation}
	\cB(M_t)= \left\{ X\subset E: |X|=t, X \supset B \text{ for some base } B \in \cB(M)\right\}.
\end{equation}
Moreover, the following proposition provides the fractional arboricity of $M_t$.
\begin{proposition}
	Let $M=(E,\cI)$ be a matroid with the rank $r(M)<|E|$. For $t=r(M),r(M)+1,\dots,|E|-1$, let $M_t$ be the matroid defined as in (\ref{eq:addedge}). Then,
	\begin{equation*}
		D(M_t)=\max\limits_{\emptyset \neq X\subset E}\frac{|X|}{t-r(M)+r(X)}.
	\end{equation*}
\end{proposition}
\begin{proof}
	For $X\subset E$, 
	\begin{align}
		r_{M_t}(X)&=|X| -r((M^*)_{|E|-t})+r_{(M^*)_{|E|-t}}(E-X)\\
		&=|X|-(|E|-t)+\min\left\{ E-t,|E|-|X|-r(M)+r(X)\right\}\\
		&=\min\left\{ |X|,t-r(M)+r(X)\right\}.
	\end{align}
If $|X| \leq t-r(M)+r(X)$, then 
\[\frac{|X|}{r_{M_t}(X)} = 1 < \frac{|E|}{t}= \frac{|E|}{r_{M_t}(E)}.\]
Hence
\begin{equation*}
	D(M_t)=\max\limits_{\emptyset \neq X\subset E}\frac{|X|}{r_{M_t}(X)} = \max\limits_{\emptyset \neq X\subset E}\frac{|X|}{t-r(M)+r(X)}.
\end{equation*}
	
\end{proof}

\begin{remark}
	Since the universal density $\eta^*$ of $M$ encodes the universal densities of all 
	truncations (Theorem~\ref{thm:main-truncation}) and of the dual matroid 
	(see~\cite{truong2024modulus}), it also encodes the universal density $\eta_t^*$ of 
	$\mathcal{M}^t$ for $t = r(M), r(M)+1, \dots, |E|-1$. In particular, $\eta^*$ also 
	determines the values of $D(M_t)$ and $S(M_t)$.
\end{remark}

 See Figure \ref{figure1} and Table \ref{table1} for an example with all universal densities of $M_t$, for $t=1,\dots,|E|$.


\section{Kullback--Leibler divergence}\label{sec:kl}
 
In this section, we study an optimization problem on the base polytope
$\co(\cB(M))$ of a matroid $M$, namely the convex hull of the indicator
functions of all bases in $\cB(M)$. 
First, we recall the relevant definitions for probability vectors (i.e., discrete distributions on finite sets). Let $p,q\in\mathbb{R}^n_{\geq 0}$ be two vectors.  The \emph{Kullback--Leibler divergence} between $p$ and $q$ is defined as
	\begin{equation*}
		D_{\KL}(p\|q) := -\sum_{i=1}^np_i\log\frac{q_i}{p_i} = - \sum_{i=1}^np_i\log q_i + \sum_{i=1}^np_i\log p_i .
	\end{equation*}
Here $\log$ denotes the logarithm of base $e$ and $0\log 0 := 0$. The first sum on the
right-hand side is the cross entropy between $p$ and $q$. This motivated Problem \cite[ Equation (6.1)]{albin2024minimizing} in the context of graphic matroids. Here we generalize it to the following problem.
Let $M(E ,\cI)$ be a matroid with weights $\si \in \R^E_{>0}$ and base family $\cB$. Let $\mathcal{N}$ be the matrix whose rows are the indicator vectors of the bases. Let $\cP(\cB)$ be the set of all probability mass functions (or pmfs) on $\cB$. Now, we consider the {\it Minimum Kullback--Leibler divergence} problem ($\MKL_{\si}(\cB)$):
\begin{equation}\label{eq:KL-beta}
	\begin{split}
		\underset{\eta,\mu}{\text{minimize}}\quad&-\sum_{e\in E}\si(e)\log\eta(e)\\
		\text{subject to}\quad&  \eta = \cN^T\mu,\\
		&\eta(e)>0 \quad \forall e\in E,\\
		& \mu\in\mathcal{P}(\cB).
	\end{split}
\end{equation}
Since the objective function in~\eqref{eq:KL-beta} diverges to $+\infty$ as any $\eta(e)$ approaches $0$, any choice of $(\eta,\mu)$ sufficiently close to the infimum lies inside a compact set and, therefore, a minimizer exists.  In fact, by the strict convexity of the objective function, the optimal $\eta$, denoted by $\eta^*$, is unique.
\begin{remark}
	Since each row sum of $\cN$ equals to $r(E)$, we have $\eta(E) = r(E).$
\end{remark}
\begin{remark}
 We recall the Gibbs's inequality for Kullback--Leibler divergence:
\begin{align*}
	-\sum_{e\in E}\si(e)\log\eta(e) &= -\sum_{e\in E}\si(e)\log\left(\frac{\eta(e)}{\eta(E)} \right)  -\sum_{e\in E}\si(e)\log\eta(E)\\
	&=\si(E)\left(-\sum_{e\in E}\frac{\si(e)}{\si(E)}\log\left(\frac{\eta(e)}{\eta(E)} \right)  \right) - \si(E)\log r(M)\\
	& \geq \si(E)\left(-\sum_{e\in E}\frac{\si(e)}{\si(E)}\log\left(\frac{\si(e)}{\si(E)} \right)  \right) - \si(E)\log r(M).
\end{align*}
	Hence, if there exists a vector $\eta \in \co(\cB)$ such that it is parallel to $\si$, meaning that $\eta = t\si$ for some $t>0$, it follows that this vector must be the unique optimal solution of $\MKL_{\si}$.
\end{remark}

The main theorem in this section is presented as follows.

\begin{theorem}\label{thm:mainmkl}
	Let $M(E ,\cI)$ be a matroid with weights $\si \in \R^E_{>0}$ and the base family $\cB$. Let $\eta^*$ be the unique optimal solution for the $\MKL_\si(\cB)$ problem. Then, $\eta^*$ equals to the universal density of $M$.
\end{theorem}
Below, we provide three different proofs for this theorem. The first proof uses the serial rule for the $\MKL$ problem, analogous to the serial rule for modulus. The other two proofs use two characterizations of the universal density.

\subsection{Proof 1 of Theorem \ref{thm:mainmkl}}
\subsubsection{The $\MKL$ problem and the core of $M$}
Let $v\in\mathbb{R}^E_{\ge 0}$ and let $B$ be a base in $\cB$, the \emph{$v$-length} of $B$ is defined as $	\ell_v(B) := \sum_{e\in B}v(e).$ First, we give the following lemma.
\begin{lemma}\label{lem:necessary-KL-beta}
	Suppose that $(\eta^*,\mu^*)$ is a minimizer for the $\MKL_{\si}$ problem (\ref{eq:KL-beta}).  Define a vector $v\in\mathbb{R}^E_{\ge 0}$ where $v(e) = \si(e)/\eta^*(e)$ for $e \in E$. Then, for all $B\in\cB$,
	\begin{equation*}
		\ell_v(B) \le \si(E).
	\end{equation*}
	If $\mu^*(B)>0$, then equality holds.
\end{lemma}

\begin{proof}
	Let $B\in\cB$ and let $\delta_B\in\mathcal{P}(\cB)$ be the Dirac mass on $B$.  For the real parameter $t$, define
	\begin{equation*}
		\mu_t = (1-t)\mu^* + t\delta_B.
	\end{equation*}
	Defining $\eta_t := \cN^T\mu_t$ yields
	\begin{equation*}
		\eta_t(e) = (1-t)\eta^*(e) + t\mathcal{N}(B,e).
	\end{equation*}
	By construction, $\mu_t$ is a pmf in $\cP(\cB)$ for any choice of $t$.  Therefore, $(\eta_t,\mu_t)$ is feasible for~\eqref{eq:KL-beta} as long as $\mu_t\ge 0$ and $\eta_t>0$.  Since $\eta^*>0$ by assumption, $\eta_t>0$ for all $t$ in a neighborhood of $0$.
	Note that
	\begin{equation*}
		\log\eta_t(e) = \log \eta^*(e) + \log \left(1-t + t \frac{\mathcal{N}(B,e)}{\eta^*(e)} \right) =  \log\eta^*(e) - t\left(1-\frac{\mathcal{N}(B,e)}{\eta^*(e)}\right) + O(t^2),
	\end{equation*}
	so
	\begin{equation*}
		-\sum_{e\in E}\si(e)\log\eta_t(e) = -\sum_{e\in E}\si(e)\log\eta^*(e)
		+ t(\si(E)-\ell_v(B)) + O(t^2).
	\end{equation*}
	If $\mu^*(B)>0$, then $(\eta_t,\mu_t)$ is feasible for $t$ in a small neighborhood of $0$, implying that $\ell_v(B)=\si(E)$.  If $\mu^*(B)=0$, then $(\eta_t,\mu_t)$ is feasible for small positive $t$, implying the inequality.
\end{proof} 
The following lemma provides the relationship between the core of $M$ and the unique optimal solution $\eta^*$ of $\MKL_\si(\cB)$.

\begin{proposition}\label{thm:vmax}
	Let $\eta^*$, $\mu^*$ and $v$ be defined as in Lemma~\ref{lem:necessary-KL-beta}.  Denote $v_{max}=\max_e v(e)$
	and let $V_{max}$ be the set of elements where $v(e)$ attains this maximum. Let $D_{\si}(M)$ be the fractional arboricity of $M$. We have the following: 
	
	\bi
	\item[(i)] \[\eta^*(V_{max}) = r(V_{max}) ;\] 
	\item[(ii)] \[v_{max} = \frac{\si(V_{max})}{r(V_{max})};\]
	\item[(iii)] $V_{max}$ is the unique maximal optimal set for $D_\si(M)$ and \[v_{max} = D_\si(M).\]
	\ei
\end{proposition}
\begin{proof}
To prove (i), by \cite[Lemma 3.12]{truong2024modulus}, we have
\begin{equation}\label{eq:lemma3.12}
r(V_{max}) \geq \eta^*(V_{max}),
\end{equation} 
and the equality holds, if and only if, every base $B \in \supp \mu^*$ satisfies $r(V_{max}) = |B \cap V_{max}|$. Suppose that  a base $B$ satisfies $|B \cap V_{max}| < r(V_{max})$. Then, by the proof of Part 1 in \cite[Theorem 4.2]{truong2024modulus}, there exist $z \in V_{max} \setminus B$ and $x \in B \setminus V_{max}$ such that $B'=(B\cup\{z\})\setminus\{x\}$ is a base.  Then,
	\begin{equation*}
		\si(E) \ge \ell_v(B') = \ell_v(B) + v(z) - v(x) > \ell_v(B).
	\end{equation*}
	By Lemma~\ref{lem:necessary-KL-beta}, we have $\mu^*(B)=0$. Therefore, (i) holds.

	For (ii), by the definition of $V_{max}$, we have
	\begin{align*}
		v_{max}\eta^*(V_{max}) = v_{max}\sum\limits_{e\in V_{max}} \eta^*(e) = &\sum\limits_{e\in V_{max}} \si(e) = \si(V_{max}).\\
	\end{align*}
Hence, by (i)
	\[v_{max} = \frac{\si(V_{max})}{\eta^*(V_{max})} = \frac{\si(V_{max})}{r(V_{max})}.\]
	
For (iii), consider any nonempty set $Z \subseteq E$ with $r(Z) > 0$, we have
	\begin{align*}
		v_{max}\eta^*(Z) = \sum\limits_{e\in Z} v_{max}\eta^*(e) \geq &\sum\limits_{e\in Z} \si(e) = \si(Z).\\
	\end{align*}
	Then, by (\ref{eq:lemma3.12}),
	\[v_{max} \geq \frac{\si(Z)}{\eta^*(Z)} \geq
	\frac{\si(Z)}{r(Z)}.\]
	Therefore, the fractional arboricity problem $D_\si(M)$ reaches its maximum at the set $V_{max}$. Note that $D_\si(M)$ attains its maximum at a unique maximal optimal set. If $Z-V_{max} \neq \emptyset$, then
	\begin{align*}
		v_{max}\eta^*(Z) = \sum\limits_{e\in Z} v_{max}\eta^*(e) > &\sum\limits_{e\in Z} \si(e) = \si(Z),\\
	\end{align*}
	and, by (\ref{eq:lemma3.12}),
	\[v_{max} > \frac{\si(Z)}{\eta^*(Z)} \geq
	\frac{\si(Z)}{r(Z)}.\]
	Hence, $Z$ is not optimal for $D_{\si}(M)$. Therefore, $V_{max}$ is the unique maximal optimal set for $D_\si(M)$.
\end{proof}

Proposition \ref{thm:vmax} shows that  $V_{max}$, defined as in Proposition \ref{thm:vmax}, is the core of $M$ with weights $\si$. Now, we provide the first proof of Theorem \ref{thm:mainmkl}. In the proof, we use the serial rule for the $\MKL$ problem, see Theorem \ref{thm:serialMKL} in the appendix, to show that the unique optimal solution $\eta^*$ of $\MKL_\si(\cB)$ preserves after contracting the core of $M$.

\begin{proof}[Proof 1 for Theorem \ref{thm:mainmkl}]
Let $\eta^*$, $\mu^*$ and $V_{max}$ be defined as in Proposition \ref{thm:vmax}.
By Theorem \cite[Theorem 3.17]{truong2024modulus}, we have
	\begin{equation*}
		\cB(M \setminus (E-V_{max})) \oplus \cB(M / V_{max}) = \cB',
	\end{equation*} 
	where $\cB'$ is the set of bases $B \in \cB$ that satisfy $r(V_{max}) = |B \cap V_{max}|$. 
	By Proof of Proposition \ref{thm:vmax}, we know that $r(V_{max}) = \eta^*(V_{max})$, and every base $B \in \supp \mu^*$ satisfies $r(V_{max}) = |B \cap V_{max}|$. Consequently, the problem $\MKL_{\si}(\cB(M))$ can be restricted to $\cB'$, and hence $\MKL_{\si}(\cB(M))=\MKL_{\si}(\cB')$. Apply Theorem \ref{thm:serialMKL} for the families $\Ga_1=\cB(M \setminus (E-V_{max})) $ and $\Ga_2=\cB(M / V_{max}) $, we obtain that the restriction of $\eta^*$ onto $E - V_{max}$ is optimal $\MKL_{\si}(\cB(M /V_{max}))$ and the value of $\eta^*(e)$ is equal to $D(M)$ for $e \in V_{max}$. Therefore, by contracting $V_{max}$ inductively, the proof is completed.
\end{proof}

\begin{remark}
	Let $\eta^*$ be the universal density of $M$ with weights $\si$. By Theorem \ref{thm:mainmkl} and Lemma \ref{lem:necessary-KL-beta}, we have
	\begin{equation*}
		\sum\limits_{e\in B}\frac{\si(e)}{\eta^*(e)} \le \si(E) \qquad \text{ for all } B\in\cB,
	\end{equation*}
	and equality holds if $B$ is supported by an optimal pmf inducing $\eta^*$.
\end{remark}

\subsection{Proof 2 and 3 of Theorem \ref{thm:mainmkl}}

We first recall some notations from \cite{fujishige1980lexicographically}. Define the {\it dependence function} $\dep$ from $\co(\cB) \times E$ into $2^E$ as follows: For an vector $x \in \co(\cB)$, $\dep(x, u)$ is the set of all elements $v \in E$ such that, for some $\epsilon > 0$, the vector $y \in \mathbb{R}^E$, defined by 
\[
y = x + \epsilon \ones_u - \epsilon \ones_v,
\]
is a vector in $\co(\cB)$. Basically, the set $\dep(x, u)$ consists of all coordinate elements of $x$ from which mass can be transported to $u$ while keeping the vector in the convex hull.

 Note that  $u \subseteq \dep(x,u)$ for any $u \in E$ where $x(u)>0$. The author in \cite{fujishige1980lexicographically} gives the following theorem.
\begin{theorem}\cite{fujishige1980lexicographically}\label{thm:lexi}
Let $x$ be a vector in  $\co(B)$, and let $\si$ be given weights in $\R^{E}_{> 0}$. Define
\[
c(e) = \frac{x(e)}{\si(e)} \quad (e \in E),
\]
and let the distinct values of $c(e)$ for $e \in E$ be given by 
\begin{equation}\label{ci}
c_1 < c_2 < \cdots < c_k.
\end{equation}
Define subsets $S_i \subseteq E$ $(i = 1, 2, \ldots, k)$ by 
\begin{equation}\label{eq:Si}
S_i = \{e \in E \mid c(e) \leq c_i\} \quad (i = 1, 2, \ldots, k).
\end{equation}
Then, the following three conditions are equivalent:
\begin{enumerate}
	\item[(i)] $x$ is the universal density
	 of $M$ with respect to the weight vector $\si$.
	\item[(ii)] $x(S_i) = r(S_i)$ for all $i = 1, 2, \ldots, k$.
	\item[(iii)] $\emptyset \subseteq \dep(x, e) \subseteq S_i$ for all $e \in S_i$ and $i = 1, 2, \ldots, k$.
\end{enumerate}
\end{theorem}
The characterizations of the universal density in (ii) and (iii) of Theorem \ref{thm:lexi} provide the second and third proof of Theorem \ref{thm:mainmkl}.
\begin{proof}[Proof 2 for Theorem \ref{thm:mainmkl}]
	Let $\eta^*$, $\mu^*$ and $v$ be as in Lemma~\ref{lem:necessary-KL-beta} and define $c(e) = 1/v(e)$ for every $e \in E$. We aim to apply Part (ii) of Theorem \ref{thm:lexi} for $\eta^*$. Define $S_i$ as in (\ref{eq:Si}). By \cite[Lemma 3.12]{truong2024modulus}, we have
	$r(S_i) \geq \eta^*(S_i)$ for every $S_i$ and the equality holds if and only if, every base $B \in \supp \mu^*$ satisfies $r(S_i) = |B \cap S_i|$. Suppose that  a base $B$ satisfies $|B \cap S_i| < r(S_i)$. Then by the proof of Part 1 in \cite[Theorem 4.2]{truong2024modulus}, there exist $z \in S_i \setminus B$ and $x \in B \setminus S_i$ such that $B'=(B\cup\{z\})\setminus\{x\}$ is a base.  Then,
	\begin{equation*}
		\si(E) \ge \ell_v(B') = \ell_v(B) + v(z) - v(x)  > \ell_v(B).
	\end{equation*}
	By Lemma~\ref{lem:necessary-KL-beta}, $\mu^*(B)=0$. Therefore, we have $\eta^*(S_i) = r(S_i)$ for all $i = 1, 2, \ldots, k$.
	By  Part (ii) of Theorem \ref{thm:lexi}, the proof is completed.
\end{proof}

\begin{proof}[Proof 3 for Theorem \ref{thm:mainmkl}]
	Let $\eta^*$, $\mu^*$ and $v$ be as in Lemma~\ref{lem:necessary-KL-beta}. We aim to apply Part (iii) of Theorem \ref{thm:lexi} for $\eta^*$. Define $c(e) = 1/v(e)$ for every $e \in E$ and let the distinct values of $c(e)$ for $e \in E$ be given by 
	\[
	c_1 < c_2 < \cdots < c_k.
	\]
Denote $f(\eta)$ the objective function of the $\MKL_{\si}(\cB)$ problem. Then $f$ is differentiable on the relative interior of $\co(\mathcal B)$, and
\begin{equation}\label{eq:deri}
\frac{\partial}{\partial \eta(e)}f(\eta)\bigg|_{\eta =\eta^*} = -\frac{\si(e)}{\eta^*(e)} = -\frac{1}{c(e)}.
\end{equation}
Let $e\in S_i$ for some $i$ and let $e'\in \dep(\eta^*,e)$.
By definition of $\dep(\eta^*,e)$, there exists $\epsilon>0$ such that
\[
\eta^*+t(\ones_e-\ones_{e'})\in \co(\mathcal B)
\qquad\text{for all }t\in[0,\epsilon].
\]
Let $d:=\ones_e-\ones_{e'}$ and define $g:[0,\epsilon]\to\mathbb R$ by
\[
g(t):=f(\eta^*+td).
\]
Then $g$ is differentiable and
\[
g'(0)
=\left.\frac{\partial f}{\partial \eta(e)}\right|_{\eta=\eta^*}
-\left.\frac{\partial f}{\partial \eta(e')}\right|_{\eta=\eta^*}
=-\frac{1}{c(e)}+\frac{1}{c(e')}
=\frac{1}{c(e')}-\frac{1}{c(e)}.
\]
If $c_i<c(e')$, then $c(e)<c_i<c(e')$, hence $g'(0)<0$.
 Therefore, for some $0<\delta<\ep$ small enough, we have
\[
f(\eta^*+\delta\ones_e-\delta\ones_{e'})<f(\eta^*),
\]
contradicting optimality of $\eta^*$.
Hence we must have $c(e')<c_i$.
So, for each $i = 1, 2, \ldots, k$, we have 
\begin{equation}\label{eq:dep}
0 \subseteq \dep(\eta^*, e) \subseteq S_i
\end{equation}
for all $e \in S_i$. 
By Theorem \ref{thm:lexi}, the proof is completed.
\end{proof}


\section{Strictly homogeneous matroid}\label{strictly-hom}

The main purpose of this section is to characterize the class of strictly homogeneous matroids, see Definition \ref{def:s-h}. First, we recall the notion of connected matroids. For a matroid $ M(E,\cI) $, a set $ X \subseteq E $ is called a {\it separator} of $ M $ if any circuit $ C \in \cC(M) $ is contained in either $ X $ or $ E - X $. A matroid $ M(E,\cI) $ is called {\it connected} if it has no separators other than $ E $ and $ \emptyset $. The following theorem is modeled after the case of graphic matroids in \cite{albin2024minimizing}. We state the result as follows.
\begin{theorem}
	Let $M=(E,\cI) $ be a matroid with given weights $\si \in\R^E_{>0}$. 
	Then, $\theta_{\si}(A) < \theta_{\si}(E)$ for every subset $A \subsetneq E$ if and only if $M$ is strictly $\si$-homogeneous and connected, where $\theta_{\si}(A) = \si(A)/r(A).$
\end{theorem}
\begin{proof}
	First, we assume that $M$ is strictly $\si$-homogeneous and connected.  It follows that $M$ is $\si$-homogeneous. Therefore, the dilation $	\beta_{\si}$ of $\si$ defined in (\ref{eq:beta-si}) satisfies $\beta_{\si} \in \co(\cB)$, and $\theta_{\si}(A) \leq \theta_{\si}(E)$ for any $A \subset E$ (since $\theta_\si(M)=D_\si(M)$). Let $\eta^*$ be the universal density of $M$ with weights $\si$. By assumption of strict homogeneity and Definition \ref{def:s-h}, there exist weights $\la\in\R_{>0}^E$ such that $\eta^*=\eta_{\la}= \cN^T{\mu_\la}$, where $\mu_\la$ and $\eta_{\la}$ are defined in (\ref{eq:mu-la}) and (\ref{eq:peredgeeffres}), and this vector is parallel with $\si$.
	 Note that  $\mu_\la(B)>0$ for every base $B \in \cB$. Suppose that $\theta_{\si}(A) = \theta_{\si}(E)$ for some $A \subsetneq E$. It follows that $\eta^*(A)=r(A)$ by the proof of Part 3 in Proposition \ref{thm:vmax}. On the other hand, since $M$ is connected, by \cite[Lemma 6.7]{truong2024modulus}, there exists $e \in A$, $e' \in E- A$ and two bases $B_1$ and $B_2$ such that \[B_1 -B_2 = \lbr e \rbr \quad \text{ and } \quad B_2 -B_1 = \lbr e' \rbr.\] Thus, $|B_2 \cap A| < |B_1 \cap A| = r(A)$, which implies that $\mu_{\la}(B_2)$=0 because of the fact $\eta^*(A)=r(A)$ and \cite[Lemma 3.12]{truong2024modulus}, contradiction. Hence, we have $\theta_{\si}(A) < \theta_{\si}(E)$ for every subset $A \subsetneq E$.
	
	\hspace{1cm}

	Next, assume that  $\theta_{\si}(A) <\theta_{\si}(E)$ for every $A \subsetneq E$. Hence, we have $\theta_{\si}(M)=D_{\si}(M)$ and it implies that  $M$ is $\si$-homogeneous. Now we want to show that $M$ is connected. If $M$ is not connected, by applying the serial rule for the base modulus to the connected components of $M$, any nontrivial separator $X$ of $M$ would satisfy $\theta_{\si}(X) = \theta_{\si}(E)$ since $M$ is $\si$-homogeneous, contradiction. The rest of the task is to prove that $M$  is strictly $\si$-homogeneous. 
	
	Let $\eta^*$ be the universal density of $M$ with weights $\si$. Suppose that $M $ is not strictly $\si$-homogeneous. By Remark \ref{rem:min-strict}, the problem (\ref{eq:min-det}) has no minimizer. By extracting a subsequence of a minimizing sequence, we may assume that there exists a constant $L>0$ and a set $E_0\subset E$ with the property that $\la_k(e)\to 0$ for all $e\in E_0$, and $\la_k(e)\ge L$ for all $k$ and $e\notin E_0$. It follows from the proof of \cite[Lemma 5.1]{albin2024minimizing} that $E_0\ne\emptyset$ since the problem (\ref{eq:min-det}) has no minimizer.  Moreover, since $\prod_e\la_k(e)^{\si(e)}=1$, we have $E_0\ne E$.  By Theorem \ref{thm:main-mindet}, we know that $\mu_{\la_k}\to\mu^*$ (element-wise convergence), where $\mu^*$ is the unique maximum-entropy pmf in $\mathcal{U}_{\beta_{\si}}$, where $\mathcal{U}_{\beta_{\si}}$ is defined in (\ref{eq:cone_definition}). 
	
	Now, we choose $H = E-E_0$. By \cite[Lemma 3.12]{truong2024modulus}, we have that $r(H) \geq \eta^*(H)$ and the equality holds if and only if, every base $B \in \supp \mu^*$ satisfies $r(H) = |B \cap H|$. Suppose that there is a base $B$ satisfying $|B \cap H| < r(H)$. By the proof of Part 1 in \cite[Theorem 4.2]{truong2024modulus}, there exist $z \in H \setminus B$ and $x \in B \setminus H$ such that $B'=(B\cup\{z\})\setminus\{x\}$ is a base.  Thus,
	
	\begin{equation*}
		\mu_{\la_k}(B) = \frac{\la_k(x)}{\la_k(z)}\mu_{\la_k}(B') \le \frac{\la_k(x)}{L} \text{ for all } k,
	\end{equation*}
	which implies that
	\begin{equation*}
		\mu^*(B) = \lim_{k\to\infty}\mu_{\la_k}(B) = 0.
	\end{equation*}
	Therefore, we obtain that $r(H) = \eta^*(H)$ and \[ \theta_{\si}(H) = \frac{1}{r(H)}\si(H) = \frac{1}{r(H)}\frac{\eta^*(H)\si(E)}{r(E)} = \frac{\si(E)}{r(E)} = \theta_{\si}(E).\]
	This is a contradiction. The proof is completed.
\end{proof}

The next theorem provides the uniqueness of $\lambda$ in Definition \ref{def:s-h} of strict homogeneity.

	\begin{theorem}
		Let $M=(E,\cI)$ be a connected matroid with base family $\cB$.  Let $\beta$ be a vector in $\co^+(\cB)$. Then there exists a vector $\lambda \in \R^E_{>0}$, unique up to independent positive rescalings on each connected component of $M$, such that $\eta_{\lambda}=\beta$, where $\eta_{\lambda}$ be defined as in (\ref{eq:peredgeeffres}). In other words, if $\lambda$ satisfies $\eta_{\lambda}=\beta$, then any $\tilde\lambda$ obtained by multiplying all weights in each connected component $C$ by a constant $\lambda_C>0$ also satisfies $\eta_{\tilde\lambda}=\beta$, and these are the only such vectors.

	\end{theorem}
	\begin{proof}
			For $\beta \in \co^+(\cB)$, by Theorem \ref{thm:supp}, there exist $\lambda^*$ that minimizes problem (\ref{eq:min-det}) (with $\si:=\beta$), and $\mu_{\lambda^*}=\mu^*_{\beta},$ where $\mu^*_{\beta}$ is the unique maximum-entropy pmf in $\mathcal{U}_\beta$.  
		
		Next, let $\lambda \in \R^E_{>0}$ be a vector such that $\eta_{\lambda}=\beta$.

		Because we always can scale $\lambda$ and it does not change $\eta_{\lambda}$, hence you can assume that 
		\begin{equation}\label{eq:cons}
			\prod\lambda(e)^{\beta(e)}=1.
		\end{equation}
		Denote  $\mu:=\mu_\lambda$ and $Z:=\sum\limits_{B\in \cB} \lambda[B]$, we want to show that:
		
		\[Z= \exp(H(\mu)),\]
		where $H(\mu)$ is defined in  (\ref{eq:max-entropy}).
		We have that (\ref{eq:cons}) is equivalent to \[\prod\limits_{B \in\cB} \lambda[B]^{\mu(B)}=1.\] Note that $\lambda[B]=Z\mu(B)$, hence 
		\[ \prod\limits_{B \in\cB} (Z\mu(B))^{\mu(B)}=1.\]
		Thus, \[ Z = \prod\limits_{B \in\cB}\mu(B)^{-\mu(B)}=\exp(H(\mu)).\]
		Therefore, \[\exp(H(\mu)) = Z \geq \sum\limits_{B\in \cB} \lambda^*[B] = \exp(H(\mu_{\beta}^*)) \geq \exp(H(\mu)).\] This implies that $\mu^*_{\beta}=\mu = \mu_{\lambda}$.
		Without loss of generality, we assume that $M$ is connected.
		If $\lambda$ satisfies $\eta_{\lambda}=\beta$, then so does $t\lambda$. Hence, it is enough to show that there exists a unique $\lambda$ such that $\lambda[B] = \mu^*_{\beta}(B) $ for all $B \in \cB$.
		By taking logarithm on both sides, the system of equations is equivalent to 
		\begin{equation}\label{eq:system}
			\cN(\log\lambda)=\log\mu^*_{\beta}.
		\end{equation}
		where $\log \lambda$ and $\log \mu^*_{\beta}$ denote taking the logarithm
		element-wise on all entries of $\lambda$ and $\mu^*_{\beta}$.
		Since $M$ is connected, by \cite[Proposition 2.4]{matroidpolytopes}, the matrix $\cN$ has rank equal to the dimension of the matroid polytope plus one, which equals to $|E|-c(M)+1=|E|$, where $c(M)$ is the number of connected components of $M$. Hence, the system of equation (\ref{eq:system}) contains at most one solution. Therefore, $\lambda$ is unique.
		
	\end{proof}

\section{Some topics in the base modulus for weighted matroids}
\subsection{Strength and fractional arboricity of a weighted matroid}\label{sec:sd}

In this section, we present several lemmas and propositions concerning 
the strength $S_{\si}(M)$ and fractional arboricity $D_{\si}(M)$ 
of a matroid $M$ with weights $\si$. 
Recall that a matroid $M$ is $\si$-homogeneous if and only if 
$S_{\si}(M) = D_{\si}(M)$. 
We also recall Parts 1 and 2 of Lemma 5.4 from \cite{truong2024modulus} 
as follows:

\begin{lemma}\cite[Lemma 5.4]{truong2024modulus}\label{lem:5sd}
	Let $M=(E,\cI)$ be a matroid with weights $\si$. Given a subset $ X \subseteq E$, the following properties hold:
	\begin{enumerate}
		\item If $ \cl( X) \neq E $, then $ S_{\si}(M) \leq S_{\si}(M / X) $.
		\item If $ X \neq E $, then $ D_{\si}(M) \geq D_{\si}(M \setminus X) $.
	\end{enumerate}
\end{lemma}
Next, we provide several properties similar to Lemma \ref{lem:5sd}.

\begin{lemma}\label{lem:dd-hom}
	Let $M=(E,\cI)$ be a matroid with weights $\si$. Assume that $M$ is $\si$-homogeneous.
	\bi
	
	\item For any set $H \subset E$ such that $\cl(H) \neq E$, we have $D_{\si}(M/H) \geq D_{\si}(M)$.
	\item For any set $\emptyset \neq H \subseteq E$, we have $S_{\si}(M|H) \leq S_{\si}(M)$.
	\ei
\end{lemma}

\begin{proof}
	We have
\[D_{\si}(M/H) \geq \frac{\si(E - H)}{r(M/H)} = \frac{\si(E - H)}{r(M) - r(H)} \geq S_{\si}(M) = D_{\si}(M),\]
and
\[
S_{\si}(M|H) \leq \frac{\si(H)}{r(H)} \leq D_{\si}(M) = S_{\si}(M).
\]
\end{proof}

\begin{lemma}\label{lem:ss-con-del}
	Let $M=(E,\cI)$ be a matroid with weights $\si$. 
	\bi
	\item For any set $H$ such that $\cl(H) \neq E$, we have $S_{\si}(M \setminus H) \leq S_{\si}(M/H)$.
	\item For any set $H$ such that $\cl(H) \neq E$, we have $D_{\si}(M \setminus H) \geq D_{\si}(M/H)$.
	\ei
\end{lemma}

\begin{proof}
	Denote $K = E - H$, we have:
	\begin{align*}
		S_{\si}(M \setminus H) 
		&= \min\limits_{X \subset K} \frac{\si(K - X)}{r_{M \setminus H}(K) - r_{M \setminus H}(X)} \\
		&= \min\limits_{X \subset K} \frac{\si(K - X)}{r(K) - r(X)}.
	\end{align*}
We also have 
	\begin{align*}
		S_{\si}(M/H) 
		&= \min\limits_{X \subset K} \frac{\si(K - X)}{r_{M/H}(K) - r_{M/H}(X)} \\
		&= \min\limits_{X \subset K} \frac{\si(K - X)}{r(E) - r(H) - r(X \cup H) + r(H)} \\
		&= \min\limits_{X \subset K} \frac{\si(K - X)}{r(E) - r(X \cup H)}.
	\end{align*}
Note that the rank function is a submodular function, then for any set $X \subset K$, \[r(K) + r(X \cup H) \geq r(K \cap (X \cup H)) + r(K \cup (X \cup H)) = r(X) + r(E).\]
	Hence, $S(M \setminus H) \leq S(M/H)$.

Next, we have
	\begin{align*}
		D_{\si}(M \setminus H) 
		&= \max\limits_{X \subset K} \frac{\si(X)}{r_{M \setminus H}(X)} \\
		&= \max\limits_{X \subset K} \frac{\si(X)}{r(X)},
	\end{align*}
and
	\begin{align*}
		D_{\si}(M/H) 
		&= \max\limits_{X \subset K} \frac{\si(X)}{r_{M/H}(X)} \\
		&= \max\limits_{X \subset K} \frac{\si(X)}{r(X \cup H) - r(H)}.
	\end{align*}
Since the rank function is a submodular function,  we have \[r(H) + r(X) \geq r(H \cap X) + r(H \cup X) = r(H \cap X).\] Therefore, \[D_{\si}(M \setminus H) \geq D_{\si}(M/H).\]
\end{proof}
The following lemma shows that the strength of a matroid $M$ remains unchanged if we contract any set  $H\subseteq (E - E_{\max})$ where $E_{max}$ is the set of  elements that attain the maximal value of the universal density of $M$, which is the maximal set that is optimal for the strength problem $S_{\si}(M)$.
\begin{lemma}
	Let $M=(E,\cI)$ be a matroid with weights $\si$. Let $E_{max}$ be the set of  elements that attain the maximal value of the universal density of $M$. Let $H$ be a subset of $E$. If $H\subseteq (E - E_{\max})$, then $S_{\si}(M/H) = S_{\si}(M)$.
\end{lemma}

\begin{proof}
Note that $E - E_{\max}$ is a closed set, see Theorem \cite[Theorem 4.1]{truong2024modulus}. We have $\cl(H) \subseteq \cl(E - E_{\max}) = E - E_{\max} \neq E $. Hence, we have
\begin{align*}
S_{\si}(M/H) &\geq S_{\si}(M) &(\text{By Lemma } \ref{lem:5sd})\\
&= S_{\si}(M/(E - E_{\max})) &(\text{By  (\ref{eq:pre}}))\\
&= S_{\si}\Bigl( M/H/\bigl( (E - E_{\max}) - H \bigr)\Bigr) &(\text{By Proposition } \ref{prop:con-del})\\
&\geq S(M/H)&(\text{By Lemma } \ref{lem:5sd}).
\end{align*}
Thus, the equality must happen.
\end{proof}
We also have a similar lemma related to fractional arboricity.
\begin{lemma}
	Let $M=(E,\cI)$ be a matroid with weights $\si$. Let $E_{min}$ be the set of  elements that attain the minimal value of the universal density of $M$. Let $H$ be a subset of $E$. If $H \supset E_{\min}$, then \(D_{\si}(M|H) = D_{\si}(M)\).
\end{lemma}

\begin{proof}
	We have
	\begin{align*}
		D_{\si}(M|H) &\geq D_{\si}(M|E_{\min}) &(\text{By Lemma } \ref{lem:5sd})\\
		&= D_{\si}(M) &(\text{By definition of } E_{min})\\
		&\geq D_{\si}(M|H)&(\text{By Lemma } \ref{lem:5sd}).
		\end{align*}
Thus, the equality must happen.
\end{proof}
Finally, we provide two main propositions of this section.
\begin{proposition}\label{prop:ss-res}
	Let $M=(E,\cI)$ be a matroid with weights $\si$. Let $H$ be a subset of $E$. If $H \cap E_{\max}\neq \emptyset$, then $S_{\si}(M|H) \leq S_{\si}(M)$.
\end{proposition}

\begin{proof}
	Denote $ A = H \cap (E - E_{\max}) $, $ B= H \cap E_{\max}$, $C = E_{\max} \setminus H $, and $ D = (E - E_{\max}) \setminus H $. It follows that $B \neq \emptyset$ and $C \neq E_{\max}$. We have the following:
	\begin{align*}
		S_{\si}(M|H) &= S_{\si}((M \setminus C) \setminus D) &(\text{By Proposition } \ref{prop:con-del})\\
		&\leq S_{\si}(((M \setminus C) \setminus D) / A) &(\text{By Lemma } \ref{lem:5sd})\\
		&= S_{\si}(((M \setminus C) / A) \setminus D) &(\text{By Proposition } \ref{prop:con-del})\\
		&\leq S_{\si}(((M \setminus C) / A) / D) & ( \text{By Lemma \ref{lem:ss-con-del},} \\
			&&\text{ and note that } B \neq \emptyset,\text{ so these two matroids are not empty})\\
		&= S_{\si}(((M / A) / D) \setminus C) &(\text{By Proposition } \ref{prop:con-del})\\
		&= S_{\si}((M / (E - E_{\max})) \setminus C) &(\text{By Proposition } \ref{prop:con-del})\\
		&\leq S_{\si}(M / (E - E_{\max})) & (\text{since } M / (E - E_{\max}) \text{ is } \si\text{-homogeneous and by Lemma \ref{lem:dd-hom}}).
	\end{align*}
\end{proof}

\begin{proposition}\label{prop:dd-con}
	Let $M=(E,\cI)$ be a matroid with weights $\si$. Let $H$ be a subset of $E$. If $E_{\min} \setminus H \neq \emptyset$ and $\cl(H) \neq E$, then $D_{\si}(M/H) \geq D_{\si}(M)$.
\end{proposition}

\begin{proof}
	Denote $ A = H \cap (E - E_{\min}) $, $ B= H \cap E_{\min}$, $C = E_{\min} \setminus H $, and $ F = (E - E_{\min}) \setminus H $. It follows that  $C \neq \emptyset$. We have the following:
	\begin{align*}
		D_{\si}(M/H) &= D_{\si}((M/A)/B) &(\text{By Proposition } \ref{prop:con-del})\\
		&\geq D_{\si}(((M/A)/B) \setminus F) &(\text{By Lemma } \ref{lem:5sd})\\
		&= D_{\si}(((M/B) \setminus F) / A) &(\text{By Proposition } \ref{prop:con-del})\\
		&\geq D_{\si}(((M/B) \setminus F) \setminus A)& ( \text{By Lemma \ref{lem:ss-con-del},}\\
			&&\text{ and note that } C \neq \emptyset,\text{ so these two matroids are not empty})\\
		&= D_{\si}(((M \setminus F) \setminus A) / B) &(\text{By Proposition } \ref{prop:con-del}) \\
		&= D_{\si}((M \setminus (E - E_{\min}))/B) &(\text{By Proposition } \ref{prop:con-del})\\
		&\geq D_{\si}(M \setminus (E - E_{\min})) & (\text{since } M \setminus (E - E_{\min}) \text{ is } \si\text{-homogeneous and by Lemma \ref{lem:dd-hom}}).
	\end{align*}
\end{proof}

\subsection{Modulus for dual matroids}\label{sec:dualmatroid}

Given a loopless matroid $ M(E,\cI) $ with $ r(E) > 0 $, let $ \cB $ be the base family of $ M $. We recall that the set
\[\cB^* = \left\{ X \subseteq E : \text{there exists a base } B \in \cB \text{ such that } X = E - B \right\}\]
is the family of bases of the dual matroid $ M^* $ of $M$ with the rank function $ r^* $. The rank function $r^*$ satisfies $ r^*(M^*) = |E| - r(M) $. The dual of the dual of a matroid $ M $ is the matroid $ M $ itself, in other words, $ (M^*)^* = M $.

In \cite{truong2024modulus}, the authors provide a relationship between the optimal solution of (\ref{eq:2norm}) of a matroid and its dual. Specifically, the optimal solution of (\ref{eq:2norm}) of $M$ and its dual add up to the vector of all ones. In this section, we generalize this result to matroids with weights as follows. 

\begin{theorem}\label{thm:dualmat}
	Let $ E $ be a finite set. Let $ \Gamma $ be a family of vector in $\R_{\geq 0}^E$ with fixed row sum $k>0$. Let $\si$ be a set of weights in $\R_{>0}^E$ such that $\si \geq \gamma$ for any $\ga \in \Ga$. We define a family of vectors $\Gamma^* := \{ \si - \gamma : \gamma \in \Gamma \} $.  Let $\widehat{\Ga}$ and $\widehat{\Ga^*}$ be the Fulkerson blocker families of $\Ga$ and $\Ga^*$, respectively. Let $ \eta^* $ and $ \eta_{\circ}^* $ be the optimal densities for $ \Mod_{2,\si}(\widehat{\Gamma}) $ and $ \Mod_{2,\si}(\widehat{\Gamma^*}) $, respectively. Then we have 
	\[\eta^* + \eta_{\circ}^* = \si.\]
\end{theorem}

\begin{proof}
Note that every vector $\ga \in \Ga$ has fixed row sum $k>0$ and $\eta^*$ belongs to the convex hull of vectors in $\Ga$, we have
	\begin{equation}\label{eq:sum-eta}
		\eta^*\cdot \one = k.
	\end{equation} 
Similarly, for $\Ga^*$, we have
	\begin{equation}\label{eq:sum-etastar}
		\eta_{\circ}^*\cdot \one = \si(E) - k.
	\end{equation}
	We define vectors $z$ and $z_{\circ}$ as follows: $ z(e) = \si(e)^{-1}\eta^*(e)$ and $ z_{\circ}(e) =\si(e)^{-1}\eta_{\circ}^*(e)$. By (\ref{eq:mod2}), we have that $\rho^* = z/\cE_{2,\si^{-1}}(\eta^*)$ is the optimal solution for $\Mod_{2,\si}(\Ga)$ and  $\rho_{\circ}^* = z_{\circ}/\cE_{2,\si^{-1}}(\eta_{\circ}^*)$ is the optimal solution for $\Mod_{2}(\Ga^*)$. 
	Denote $a = \si- \eta_{\circ}^*$ and define $z_a$ where $z_a(e) = \si(e)^{-1}a(e)$. Hence, $ z_a = \one - z_{\circ}$.
	We have, by admissibility, for any $\ga \in \Ga$,

	\begin{align*}
		&\left( \si -\ga \right) \cdot z_{\circ} \geq \cE_{2,\si^{-1}}(\eta_{\circ}^*) \\
		& \Leftrightarrow 
		\left( \si -\ga \right) \cdot (\one - z_a) \geq \cE_{2,\si^{-1}}(\si - a)  \\
		& \Leftrightarrow \si(E)-k -k + \ga\cdot z_a \geq \si(E) - 2k + \cE_{2,\si^{-1}}(a) \\
		&\Leftrightarrow \ga\cdot z_a \geq \cE_{2,\si^{-1}}(a).
	\end{align*}
	In other words, the vector \[\frac{z_a}{\cE_{2,\si^{-1}}(a)} =  \frac{\one -z_{\circ}}{\cE_{2,\si^{-1}}(\si-\eta_{\circ}^*)}\] is admissible for $\Ga$.  Hence, since $\rho^*$ is optimal for $\Mod_2(\Ga)$, we have
	\begin{align}\label{eq:e1}
		\frac{1}{\cE_{2,\si^{-1}}(\si -\eta_{\circ}^*)} &= \cE_{2,\si} \left(\frac{\one -z_{\circ}}{\cE_{2,\si^{-1}}(\si-\eta_{\circ}^*)} \right) \geq \cE_{2,\si}(\rho^*) = \frac{1}{\cE_{2,\si^{-1}}(\eta^*)}.
	\end{align}
	Similarly, $\frac{\one -z}{\cE_{2,\si^{-1}}(\si-\eta^*)}$ is admissible for $\Ga^*$ and
	\begin{align}\label{eq:e2}
		\frac{1}{\cE_{2,\si^{-1}}(\si -\eta^*)}  \geq \frac{1}{\cE_{2,\si^{-1}}(\eta_{\circ}^*)}.
	\end{align}
	These two inequalities are equivalent to
	\begin{align*}
		\cE_{2,\si^{-1}}(\eta^*) \geq \cE_{2,\si^{-1}}(\si -\eta_{\circ}^*) &= \si(E) -2(\si(E)-k) +\cE_{2,\si^{-1}}(\eta_{\circ}^*),& (\text{by (\ref{eq:sum-etastar})})
	\end{align*}
	and 
	\begin{align*}
		\cE_{2,\si^{-1}}(\eta_{\circ}^*) \geq \cE_{2,\si^{-1}}(\si -\eta^*) &= \si(E) -2k +\cE_{2,\si^{-1}}(\eta^*).& (\text{by (\ref{eq:sum-eta})})
	\end{align*}
	Consequently, we have shown that		
	\begin{equation*}
		\cE_{2,\si^{-1}}(\eta^*)  - k \geq \cE_{2,\si^{-1}}(\eta_{\circ}^*) -(\si(E)-k) \geq	\cE_{2,\si^{-1}}(\eta^*)  - k.
	\end{equation*}
Therefore, two inequalities in (\ref{eq:e1}) and (\ref{eq:e2}) holds as equalities. Hence, by uniqueness of $\eta^*$ and $\eta_{\circ}^*$, we obtain that $\si - \eta_{\circ}^* = \eta^*$.	
\end{proof}
\subsection{Union of matroids with rank 1}\label{sec:pmf}
Let $M$ and $N$ be two uniform matroids of rank $1$ with the ground set $A$ and $B$, respectively where $A$ = $\{ a_1,a_2,a_3,\dots,a_n\}$ and $B= \{ b_1,b_2,b_3,\dots,b_n\}$. Let $E =A \cup B$.  Let $M \oplus N$ be the matroid direct sum of $M$ and $N$ defined as in (\ref{eq:directsum}). Let $\Ga$ be the base family of $M \oplus N$ with usage vectors being indicator functions. Then \[\Ga = \{ \ga_{a_ib_j} := \{a_i,b_j\}:i,j =1 \dots n\},\] with usage vectors being indicator functions. It is straightforward that $M$ and $N$ are homogeneous. By the serial rule for modulus, $M \oplus N$ is also homogeneous with the universal base $\eta^*(e) = 1/n$ for every $e\in E$. In this section, we study the set of all optimal pmfs $\mu^*$ inducing $ \eta^* $. 

Let $F$ be the set of all bijections $f$ from $A$ to $B$. For each $f \in F$, we define a pmf $\mu_f \in \cP(\Ga)$ which is uniform on the family $\{ \{a_i,f(a_i)\}:i=1,\dots,n\}$. Then, every $\mu_f$ induces $ \eta^* $. Next, we state our main result in this section, this result describes the set of all optimal pmfs inducing $ \eta^* $.
\begin{theorem}\label{thm:pmf-conv}
	A pmf $\mu \in \cP(\Ga)$ induces $ \eta^* $ if and only if it belongs to $\co(\{\mu_f:f \in F\})$.
\end{theorem}
\begin{proof}
	
	First, for every $f \in F$, the pmf $\mu_f$ induces $ \eta^* $. Hence, any convex combination of pmfs $\mu_f$ induces $ \eta^* $.
	
	Next, let $\mu \in \cP(\Ga)$ induce $ \eta^* $. Denote $ u_{a_ib_j}:= \mu(\ga_{a_ib_j})$ for $i,j=1,\dots,n$. Hence,
	\begin{equation*}
		\mu = \sum\limits_{i,j=1}^n u_{a_ib_j}\delta_{\ga_{a_ib_j}},
	\end{equation*}
	where  $\delta_{\ga_{a_ib_j}}$ is the indicator vector in $\R^{\Ga}_{\geq 0}$ of $\ga_{a_ib_j}$
	Without loss of generality, assume that $u_{a_1b_1}$ is the smallest nonzero coefficient. We want to show that there exists a map $f\in F$ such that $f(a_1)=b_1$ and $u_{a_if(a_i)} > 0$ for $i=1,\dots,n$.  Fix $i > 1$, note that $\eta^*(a_i)=\eta^*(b_1)$ for all $i=1,\dots,n$, we have 
	\begin{align*}
		u_{a_ib_1} +\sum\limits_{j=2}^n u_{a_ib_j} = \eta^*(a_i)=\eta^*(b_1)\geq u_{a_ib_1} +u_{a_1b_1} > u_{a_ib_1} +0
	\end{align*}
	Therefore, for every $i > 1$,
	\begin{align*}
		\sum\limits_{j=2}^n u_{a_ib_j} >0.
	\end{align*}
	Similarly, for every $j > 1$,
	\begin{align*}
		\sum\limits_{i=2}^n u_{a_ib_j} >0.
	\end{align*}
	Hence,
	\begin{equation}\label{eq:bigprod}
		\prod_{i = 2}^n \left(\sum\limits_{k=2}^n u_{a_ib_k} \right)\prod_{j = 2}^n \left(\sum\limits_{h=2}^n u_{a_hb_j} \right) >0.
	\end{equation}
	We distribute this product, there exists at least a positive term, denoted by $x$, in the summation. 	This term $x$ is a product of $(n-1)^2$ numbers of the form $u_{a_ib_j}$. 
	We claim that there exists $(n-1)$ numbers among those $(n-1)^2$ numbers that forms a set $\{u_{a_if(a_i)}: i=2,\dots,n \}$ for some $f \in F$. Assume that our claim is true, we have $u_{a_if(a_i)} >0$ for $i=2,\dots,n$. We rewrite \[\mu =  u_{a_1b_1}\sum\limits_{i=1}^n \delta_{\ga_{a_if(a_i)}} + \mu-u_{a_1b_1}\sum\limits_{i=1}^n \delta_{\ga_{a_{i}f(a_i)}}.\]
	Note that the vector $\mu_1 := \mu-u_{a_1b_1}\sum\limits_{i=1}^n \delta_{\ga_{a_if(a_i)}}$ is a nonnegative vector since $u_{a_if(a_i)} \geq u_{a_1b_1} $ for all $i$. Furthermore, the sum of this vector equals $1 - nu_{a_1b_1}$. Hence, the vector $\mu_1/(1 - nu_{a_1b_1})$ is a pmf in $\cP(\Ga)$. Moreover, it is straightforward that this vector induces $ \eta^* $ and it has less number of nonzero numbers than $\mu$. We apply this argument iteratively, this process ends when we get the zero vector. Therefore,  $\mu$ can be written as a convex combination of multiple vectors $\mu_f$ for $f \in F$.
	
	It remains to prove the claim, we construct a bipartite graph $G=(L,R,E_G)$ where $L$ = $\{ a_2,a_3,\dots,a_n\}$ and $B= \{b_2,b_3,\dots,b_n\}$. For each $u_{a_ib_j}$ appearing in $x$, we add an edge connecting $a_i$ and $b_j$ to $G$ and we allow multiedges in $G$. Thus, $|E_G| = ( n-1)^2$. By (\ref{eq:bigprod}), for each $i$, $a_i$ appears in $x$ at most $1+(n-1) =n$ times and for each $j$, $b_j$ appears in $x$ at most $(n-1)+1 =n$ times. Hence, the degree of each vertex in $G$ is at most $n$. 
	Now, by applying Theorem \ref{theo:main} which is stated below, $G$ has a perfect matching. This demonstrates our claim.
\end{proof}
To prove Theorem \ref{theo:main}, we recall Hall’s Marriage Theorem.
\begin{theorem}[Hall’s Marriage Theorem]
	Let $G = (L, R, E_G)$ be a bipartite graph with $|L| = |R|$ and have multiedges. Then, $G$ has a perfect matching if and only if for every 
	$S \subseteq L$, we have $|N(S)| \geq |S|$ where \[N(S):=\{v \in R: \exists u \in S \text{ such that } \{u, v\} \in E\}\] is its set of neighbors.
\end{theorem}
Finally, we provide Theorem \ref{theo:main} which is used in the proof of Theorem \ref{thm:pmf-conv}.
\begin{theorem}\label{theo:main}
	Let $G = (L, R, E_G)$ be a bipartite graph with $|L| = |R| = n \geq 1$ and have multiedges. Suppose that the degree of each vertex in $G$ is at most $n+1$ and $|E_G| = n^2$. Then, $G$ has a perfect matching.
\end{theorem}
\begin{proof}
	Assume that $G$ does not have a perfect matching. Then, by Hall’s Marriage Theorem, there exists $S \subseteq L$ such that $|N(S)| \leq |S|-1$. Let $A$ be the set of edges connecting $S$ and $N(S)$. Let $B = E_G-A$. 
	Then 
	\begin{align*}
		(|S|-1)(n+1) & \geq |N(S)|(n+1) &\\
		& \geq |A|&(\text{Each vertex degree is at most $n+1$}) \\ 
		& = |E_G| - |B|&\\
		& \geq n^2 - (n-|S|)(n+1)&(\text{Each vertex degree is at most $n+1$}) 
	\end{align*}
	This is equivalent to 
	\[|S|(n+1) - n -1 \geq n^2 -n^2-n +|S|(n+1),\]
	which is equivalent to $-1 \geq 0$, a contradiction.
\end{proof}

\section{Two optimization problems for matroid structures}
\subsection{Removing elements with respect to having k edge-covering disjoint independent sets in a matroid}\label{sec:removing}
Let $M=(E,\cI)$ be an unweighted matroid. Let $a(M)$ denote the minimum number of bases whose union equals $E(M)$. In this section, we study the following problem: Given a matroid $M$ with $a(M) > k$, how many elements should be deleted from $M$ so that the resulting matroid $M'$ has $a(M') \leq k$?  
Given a matroid $M$ and an integer $k \geq 1$, we define $N(M, k)$ to be the maximum integer $l > 0$ such that $M$ 
has a restriction $X$ of $M$ such that $|X| = l$ and $a(M|X)\leq k$. Note that if $a(M) \leq k$, then we have $X=E(M)$.
The main purpose of this section is to determine $N(M, k)$ in terms of other related concepts of $M$. We start with the following lemma.

\begin{lemma}\label{lem:seriala}
	Let  $M=(E,\cI)$ be a matroid. Let $X $ be a subset of $E$. If $a(M/X) \leq k$ and $a(M|X) \leq k$, then $a(M) \leq k$.
\end{lemma}

\begin{proof}
	We have that $M/X$ has $k$ bases covering $E-X$ and $M|X$ has $k$ bases covering $X$. Note that the union of a base of $M/X$ and  a base of $M|X$ is a base of $M$. So, $M$ has $k$ bases covering $E$.
\end{proof}

The next lemma describes the rank of a set $X$ at which $N(M,k)$ attains its maximum.

\begin{lemma}\label{lem:ranks}
	Given a matroid $M$ and a positive integer $k$, let $X \subset E$ be a set such that  $N(M,k)$ attains its maximum at $X$. Then $r(M|X)=r(M)$.
\end{lemma}
\begin{proof}
	 Since $a(M|X) \leq k$, by Theorem \ref{edmonds-SD}, the matroid $M|X$ has $k$ bases $B_1,\dots, B_k$ whose union equals $X$. Hence,  $B_1$ is independent in $M$. Assume that $r(B_1)<r(M)$, there exist an element $e \in E$ such that $r(B_1 \cup \lbr e \rbr) > r(B_1)$. Since $r(B_1)=r(X)$, so $e \notin X$. Consider $X'=X \cup  \lbr e \rbr$, we have that $M|X'$ has $k$ independent sets $B_1 \cup \lbr e \rbr,B_2,\dots , B_k$ whose union equals $X'$. This contradicts with the optimality of $N(M,k)$.
\end{proof}

\begin{lemma}\label{lem:eta-and-k}
	Given a matroid $M$ and a positive integer $k$ such that $a(M) > k$. If $k \leq S(M)$, then 
	\begin{equation}
		N(M,k) = kr(M).
	\end{equation}
\end{lemma}
\begin{proof}
	
	Let $X \subset E$ be a set at which $N(M,k)$ attains its maximum. Thus, $a(M|X) \leq k$, and $X$ can be covered by $k$ bases of $M|X$. So, $N(M,k)=|X| \leq kr(M|X) = kr(M)$ by Lemma \ref{lem:ranks}.
	
	Since $k \leq S(M)$, by Theorem \ref{edmonds-SD}, $M$ has $k$ disjoint bases $B_1,\dots, B_k$. Let $X = \bigcup_{i=1}^{k}B_{i}$. By choosing       a uniform pmf supported by $B_1,\dots, B_k$, we have that $M|X$ is homogeneous. Hence,  \[S(M|X)=D(M|X)=a(M|X)=k.\] So, $N(M,k) \geq kr(M)$. In conclusion,  $N(M,k) = kr(M)$.
\end{proof}
The following theorem is the main theorem in this section.
\begin{theorem}\label{thm:eta-and-k}
			Given a matroid $M$ and a positive integer $k$ such that $a(M) > k$ (equivalently, $D(M)>k$).
			Let $\cB$ the base family of $M$ and $r$ be the rank function of $M$.
			Let $\eta^*_1 <\eta^*_2 <\dots <\eta^*_m$ be $m$ distinct values of the universal density $\eta^*$ of $M$.
			Define subsets $S_i \subseteq E$ $(i = 1, 2, \ldots, m)$ by 
			\begin{equation}\label{eq:Si2}
				S_i = \{e \in E \mid \eta^*(e) \leq \eta^*_i\} \quad (i = 1, 2, \ldots, m).
			\end{equation}
Note that $\eta^*_1 = 1/D(M) <1/k$. Let $i(k)$ denote the largest subscript such that $\eta^*_{i(k)} \leq 1/k.$ Then
	\begin{equation}\label{eq:eta-and-k}
		N(M,k)=  |E(M)-S_{i(k)}|+kr(S_{i(k)}).
	\end{equation}
\end{theorem}
\begin{proof}
First, note that if $1/S(M) = \eta^*_m  \leq 1/k$, then  $S_{i(k)} = E$ and the equation (\ref{eq:eta-and-k}) follows from Lemma \ref{lem:eta-and-k}.

Next, we assume that $1/S(M) = \eta^*_m > 1/k$. Then $S(M)<D(M)$, and $M$ is not homogeneous and $m>1$.
Recall that $i(k)$ is the largest subscript such that $\eta^*_{i(k)} \leq 1/k.$ By matroid decomposition generated by the serial rule for modulus of bases of matroids (see \cite{truong2024modulus}), we have \[D(M|S_{i(k)}) = 1/\eta^*_1 = D(M)>k, \quad S(M|S_{i(k)}) = 1/\eta^*_{i(k)} \geq k,\] and \[D(M/S_{i(k)}) = 1/\eta^*_{i(k)+1}<k,\quad S(M/S_{i(k)}) = 1/\eta^*_m = S(M).\] 
Hence, $a(M|S_{i(k)}) >k$.  Denote $M' =  M|S_{i(k)}$, by Lemma \ref{lem:eta-and-k}, $N(M',k) = kr(S_{i(k)})$. Let $X \subset E$ be a set at which  $N(M,k)$ attains its maximum, we have
\begin{align*}
	a\bigl(M'|(X \cap S_{i(k)})\bigr) &= a\bigl((M|X)|(X \cap S_{i(k)})\bigr) &\\
	&\leq a(M|X) & (\text{by Lemma \ref{lem:5sd}})\\
	&\leq k &(\text{by the definition of } X).
\end{align*} 
 This implies that $|X \cap S_{i(k)}| \leq N(M',k)$. Let $X' \subset S_{i(k)}$ be a set at which $N(M',k)$ attains its maximum, then
 \begin{align*} 
 	|X| &= |X \cap (E-S_{i(k)})| +  |X \cap S_{i(k)}|\\
 	&\leq |E-S_{i(k)}| + |X \cap S_{i(k)}|\\
 	&\leq  |E-S_{i(k)}| + N(M',k)\\
 	&= |E-S_{i(k)}| + |X'|\\
 &	= |E-S_{i(k)}| + kr(S_{i(k)}).
 \end{align*}
 Define $Y:= (E-S_{i(k)}) \cup X'$ and $M'':=M|Y$. We will show later that $a(M'') \leq k$. By definition,  $M''|X' = M|X' = M'|X'$. Hence, we have $a(M''|X')= a(M'|X') \leq k$. By Lemma \ref{lem:ranks} for $N(M',k)$, we have $r(X')=r(M')=r(S_{i(k)})$. Using this fact, we claim that 
 \begin{equation}\label{cl:1}
 	M|\bigl((E-S_{i(k)})\cup X'\bigr)/X'=  M/S_{i(k)}.
 \end{equation}
 To prove this claim, let $B$ be a maximal independent set in $X'$. Since $r(X')=r(S_{i(k)})$, $B$ is a base of $M|S_{i(k)}$. 
 By definition of matroid contraction, we have that any base $B_1$ of $M/S_{i(k)}$ satisfies $B_1 \subset E - S_{i(k)}$ and $B_1 \cup B$ is a base of $M$. In other words, $B_1$ satisfies $B_1 \subset (((E-S_{i(k)})\cup X') - X')$, $B_1 \cup B$ is a base of $M$, and $B$ is a base of $M|X'$. Thus, $B_1$ is a base of  $(M|((E-S_{i(k)})\cup X')/X'$. 
Therefore, two matroids $(M|((E-S_{i(k)})\cup X')/X'$ and $M/S_{i(k)}$ have the same ground set and the same base family, and hence, they are identical. So, our claim (\ref{cl:1}) is true and this implies that \[D(M''/X')= D(M/S_{i(k)}) =  1/\eta^*_{i(k)+1} <k.\] Thus, $a(M''/X') \leq k$. By using Lemma \ref{lem:seriala} for $M''/X'$ and $M''|X'$, we have that $a(M'') \leq k$. Therefore, \[|E-S_{i(k)}| + kr(S_{i(k)}) = |E-S_{i(k)}| + |X'| = |(E-S_{i(k)}) \cup X'| \leq N(M,k).\]
 In conclusion, we obtain the equality (\ref{eq:eta-and-k}).
\end{proof}

\subsection{Addable edges to graphs}\label{addable}
In the last section, we wanted to remove some elements from a matroid in order to decrease the matroid's arboricity. In this section, we focus on a different problem with the case of graphic matroid.
Let $G=(V,E)$ be an undirected connected graph. Recall that the arboricity $a(G)$ is the minimum number of spanning trees whose union equals $E(G)$. In this section, we study the question: Given a graph $G$ with $a(G) \leq k$, what kind of new edges $e$ can be added to $G$ so that $a(G+e) \leq k$? Denote $E_k(G)$ the set of edges $e = (v_1,v_2)$ where $v_1,v_2 \in V$ that can be added to $G$ such that $a(G+e) =k.$ If $a(G) = 1$, then $G$ is a spanning tree, and the addition of any edge would strictly increase its arboricity of $G$. From now on, we assume that $a(G)=k \geq 2.$
	Let's start with some notations. Let $H$ be a subgraph of $G$ such that every connected component of $H$ is a vertex-induced subgraph of $G$, we denote $G/H$ a graph obtained by identifying vertices in each connected component of $H$ in $G$ and removing self-loops afterwards. Now, we provide the serial rule for the arboricity of graphs, which is a corollary of Lemma \ref{lem:seriala}.
\begin{lemma}\label{lem:kforests}
	Given a connected graph $G$ and a positive integer $k$, let $H$ be a subgraph of $G$ such that every connected component of $H$ is a vertex-induced subgraph of $G$. If $a(G/H) \leq k$ and $a(H) \leq k$, then $a(G) \leq k$. 
\end{lemma}
\begin{proof}
	The proof is completed by applying Lemma \ref{lem:seriala} for the graphic matroid associated with $G$.
\end{proof}
\begin{proposition}\label{prop:adde}
	Given a connected graph $G$ and a positive integer $k$, let $H$ be a subgraph of $G$ such that every connected component of $H$ is a vertex-induced subgraph of $G$. Assume that $a(G/H) \leq k$ and $a(H) \leq k$. If $E_k(H) = \{ e:=(v_1,v_2):v_1,v_2 \in V(H)\}$ and $E_k(G/H) = \{ e:=(v_1,v_2):v_1,v_2 \in V(G/H)\}$, then $E_k(G) = \{ e:=(v_1,v_2):v_1,v_2 \in V(G)\}$.
\end{proposition}

\begin{proof}
	Let $ e=(v_1,v_2)$ for some $v_1,v_2 \in V(G)$. If $v_1,v_2$ do not belong to any common connected component of $H$, then $(G+e)/H = (G/H)+e$, this implies $a((G+e)/H)=a((G/H)+e)\leq k$. Note that $ a(H)\leq k$, by Lemma \ref{lem:kforests}, the graph $G+e$ satisfies $a(G+e)\leq k$.
	
	If $v_1,v_2$ belong to some common connected component of $H$, then $ a(H+e)\leq k$. In contrast, $(G+e)/(H+e) = G/H$, this implies $a((G+e)/(H+e))=a(G/H)\leq k$. By Lemma \ref{lem:kforests}, the graph $G+e$ satisfies $a(G+e)\leq k$.
\end{proof}

Now, we are ready to introduce the main theorem in this section.

\begin{theorem}
Let $G=(V,E)$ be a connected graph with $a(G) = k \geq 2$.
Let $E_k(G)$ be the set of edges $e = (v_1,v_2)$ that can be added to $G$ such that $a(G+e) =k$, where $v_1,v_2 \in V$. Let $D(G)$ be the fractional arboricity of $G$. Each of the following holds:
\bi
\item[(i)] If $D(G) < k$, then $E_k(G) = \{ e=(v_1,v_2):v_1,v_2 \in V\}$.

\item[(ii)] If $D(G) = k$ and $G$ is homogeneous, then $E_k(G) = \emptyset$.

\item[(iii)] If $D(G) = k$ and $G$ is not homogeneous, then $E_k(G)$ is the set of edges $e=(v_1,v_2)$ where $v_1,v_2$ do not belong to any common connected vertex-induced subgraph optimizing $D(G)$.
\ei
\end{theorem}
\begin{proof}
For part (i), assume that (i) is false and  let $G$ be a counter example such that $r(G)$ is minimized. By the definition of $G$, $D(G)<k$, and there exist $v_1,v_2 \in V$ such that the edge $e_1 = (v_1,v_2)$ satisfies $a(G+e_1) \geq k+1$. 

We claim that $G$ is homogeneous. Assume $G$ is not homogeneous. Let $E_{max}$ be the set of elements that attain the maximal value of the universal base $\eta^*$ of the graphic matroid associated with $G$. Let $H$ be the nontrivial subgraph which is edge-induced by $E-E_{max}$. Then, every connected component of $H$ is a vertex-induced subgraph of $G$. By the serial rule for base modulus, the graph $G/H$ is homogeneous and  \[D(G/H) = \theta(G/H) = S(G/H) = S(G) \leq \theta(G) \leq D(G)<k.\]
By the minimization of $r(G)$, we have $D(G/H+e) \leq a(G/H+e) \leq k$ for any edge $e = (v_1,v_2)$ where $v_1,v_2 \in V(G/H)$.
Let $H_1$ be an arbitrary connected component of $H$.
By the construction of $H_1$ and the serial rule with $E_{max}$, we have that $D(H_1)=D(G)<k$. Then, by the optimality of $r(G)$, we have that $D(H_1+e) \leq a(H_1+e) \leq k$ for any edge $e = (v_1,v_2)$ where $v_1,v_2$ belong to $V(H_1)$. Furthermore, for any  edge $e = (v_1,v_2)$ where two vertices $v_1,v_2$ belong to different components of $H$, we have $a(H+e)=a(H)\leq k$.
Therefore, by Proposition \ref{prop:adde}, $E_k(G) = \{ e=(v_1,v_2):v_1,v_2 \in V\}$, which is a contradiction. 
Hence, $H$ must be $G$ and $G$ is homogeneous.

 Since $G$ is homogeneous, we have \[\theta(G)=D(G) = \frac{|E(G)|}{r(G))} <k.\] So, $|E(G)| <r(G)k$. Hence, $|E(G)| \leq r(G)k-1$.  Thus,  $|E(G)|+1 \leq kr(G)$. Therefore, \[|E(G+e_1)| \leq r(G)k = r(G+e_1)k \Rightarrow \theta(G+e_1) \leq k.\]
 
Next, we claim that $G+e_1$ is homogeneous. Assume that $G+e_1$ is not homogeneous. Let $K$ be the nontrivial subgraph which is edge-induced by the set of edges that do not attain the maximal value of the universal base of $G+e_1$. Then, every connected component of $K$ is a vertex-induced subgraph of $G+e_1$. It follows that $e_1 \in K$, because otherwise, $k+1 \leq D(G+e_1) = D((G+e_1)|K)=D(G|K) \leq D(G) <k$, contradiction. Furthermore, $(G+e_1)/K$ is homogeneous and \[D((G+e_1)/K) = \theta((G+e_1)/K) = S((G+e_1)/K) = S(G+e_1) \leq \theta(G+e_1) \leq k.\] 
Hence, $a((G+e_1)/K) \leq k$. On the other hand, we have $D(K-e_1) \leq D(G)<k$. By the optimality of $r(G)$, it follows that $a(K) \leq k$. By Lemma \ref{lem:kforests}, $a(G+e_1) \leq k$, contradiction. Therefore, $G+e_1$ must be homogeneous. 

Since $G+e_1$ is homogeneous, by $ \theta(G+e_1) \leq k$, we have  $a(G+e_1) \leq k$, contradiction. 

For (ii), since  $S(G)=\theta(G) = D(G) = k$, we have that $|E(G)| = r(G)k$. If we add an arbitrary edge $e = \{v_1,v_2\}$ to the graph $G$, then any $k$ forests of $G+e$ contains at most $r(G)k <r(G)k+1 = |E(G+e)| $. Hence, $a(G+e) \geq k+1$. Thus, $E_k(G) = \emptyset$.

For (iii), let $E_{max}$ be the set of elements that attain the maximal value of the universal base $\eta^*$ of $G$, and let $H$ be the nontrivial subgraphs with possibly more than one connected components which is edge-induced by the set of elements that attain minimal value of $\eta^*$. Each connected vertex-induced component $K$ of $H$ is optimal for $D(G)$. It follows that $D(K)=D(G)=k$. Let $e=(v_1,v_2)$ be an arbitrary edge where  $ v_1,v_2 \in V(K)$, by (ii), $a(K+e) \geq k+1$. Hence $a(G+e) \geq a(K+e) \geq k+1 $. It follows that $e \notin E_k(G)$.
Let $e=(v_1,v_2)$ be an arbitrary edge such that $v_1,v_2$ do not belong to any common connected vertex-induced component of $H$.  It follows that 
\[a((G+e)|H) = a(G|H) = D(G|H) = D(G)=k.\] On the other hand, we have $D(G/H)< D(G)=k$. By (i), $a((G+e)/H)= a(G/H+e)  \leq k$. Hence, by Lemma \ref{lem:kforests}, $a(G+e) \leq k$. The proof for (iii) is completed.

\end{proof}

\bibliographystyle{acm}
\bibliography{ppfinal}
\def\cprime{$'$}
\nocite{*}

\appendix

\section{Appendix: Matroids}\label{subsec:mat}
Let us begin by revisiting several definitions related to matroids. 
For a set $X$ we write $|X|$  for its cardinality, and if $Y$ is another set, then $X-Y$ is the relative complement of $Y$ in $X$.

\begin{definition}\label{def:independent-set} 	Let $E$ be a finite set, the set system $M(E,\cI)$ is a matroid if the  following axioms are satisfied:
	\bi
	\item[(I1)] $\emptyset \in \cI$.
	\item[(I2)] If $X \in \cI$ and $Y \subseteq X$ then $Y \in \cI$ ({\it Hereditary property}).
	\item[(I3)] If $X,Y \in \cI$ and $|X| > |Y|$, then there exists $x \in X - Y$ such that $Y \cup \left\{x	\right\} \in \cI$ ({\it Exchange property}).
	\ei
	Every set in $\cI$ is called an {\it independent set}.
\end{definition}
The maximal independent sets are called {\it bases},  the minimal dependent sets are called {\it circuits}. 
The {\it rank} function, $r : 2^E \rightarrow \mathbb{Z}_{+}$, defined on all subsets $X\subset E$ is given by:
\[r(X) := \max \left\{ |Y| : Y \subseteq X, Y \in \cI \right\}.\]

Next, we recall the definition of dual matroids. Given a matroid $M(E,\cI)$, the set
\[ \cB^* :=\{ E-B: B \in \cB \}\]
is the family of bases of the {\it dual matroid} $M^*$ on $E$. The {\it corank} function $r^*$ of $M$ is defined as the rank function of $M^*$, and for any subset $X \subseteq E$, we have:
\begin{equation}\label{eq:corank-rank}
	r^*(X) = |X| - r(M) + r(E-X).
\end{equation}

Let us also recall the operations of {\it deletion}, {\it restriction}, and {\it contraction} in matroids. For a matroid $M(E,\cI)$ and a subset $X \subseteq E$, the {\it matroid deletion} of $X$ from $M$ is denoted by $M \setminus X$ and its collection of independent sets is
\[ \cI(M \setminus X) := \{ Y \subseteq E -X : Y \in \cI(M)\}.\]
The {\it restriction} to $X$ in $M$ is denoted by $M|X$, and is the matroid on $X$ defined as $M|X := M \setminus (E - X)$. The base family of $M\setminus X$ is the collections of maximal sets in $\{  B - X : B \in \cB(M) \}$. Its rank function is $r_{M \setminus X}(Y) = r_{M}(Y)$ for all subsets $Y \subseteq E-X$.
Note that $r_M$ denotes the rank function of a matroid $M$. When we discuss a matroid $M$ and its deletion and contraction, $r$ is understood as the rank function of $M$.
The {\it matroid contraction} of $X$ in $M$ is denoted by $M/ X$ and its collection of independent sets is
\[\cI(M / X) = \lbr Y \subseteq E - X :\exists B \in \cB(M \setminus(E - X)) \text{ such that } Y \cup B \in \cI(M)\rbr.\] The base family of $M/ X$ is $\cB(M / X) = \lbr Y \subseteq E - X : \exists B \in \cB(M \setminus(E - X))\text{ such that }Y \cup B \in \cB(M) \rbr,$ and the rank function is $r_{M / X}(Y) = r_{M}(X \cup Y) - r_{M}(X)$ for all $Y \subseteq E - X$.
For simplicity, throughout this paper, we only consider loopless matroids with positive rank. Loopless means that $r(X)=0$ implies $X=\emptyset$ and positive rank means that $r(E)>0$.
We recall more properties in the following propositions.
\begin{proposition}\cite{mat}\label{prop:con-del}
	For a matroid \(M\) and disjoint \(X, Y \subseteq E(M)\), we have:
	\begin{enumerate}
		\item \((M \setminus X) \setminus Y = M \setminus (X \cup Y)\).
		\item \((M / X) / Y = M / (X \cup Y)\).
		\item \((M / X) \setminus Y = (M \setminus Y) / X\).
	\end{enumerate}
\end{proposition}

\section{Appendix: Discrete Modulus}\label{subsec:mod}

Let $E$ be a finite set with given weights 
$\sigma \in \mathbb{R}^E_{>0}$ assigned to 
elements in $E$. We say that $\Gamma$ is a family 
of objects in $E$ if, for each object $\gamma \in \Gamma$, 
there is an associated function  
$\mathcal{N}(\gamma, \cdot)^T: E \to \mathbb{R}_{\geq 0}$, 
which we interpret as a {\it usage vector} in 
$\mathbb{R}^E_{\geq 0}$. In other words, $\Gamma$ is 
associated with a $|\Gamma| \times |E|$ 
{\it usage matrix} $\mathcal{N}$. 
We assume that $\Gamma$ is nonempty and 
that each object $\gamma \in \Gamma$ uses at least one 
element in $E$ with a positive and finite amount.

A {\it density} $\rho$ is a vector in 
$\mathbb{R}^E_{\geq 0}$. For each $\gamma \in \Gamma$, 
we define the {\it total usage cost} 
$\ell_{\rho}(\gamma)$ of $\gamma$ with respect to $\rho$ as
\[
\ell_{\rho}(\gamma) = \sum\limits_{e \in E} 
\mathcal{N}(\gamma, e) \rho(e).
\]
A density $\rho \in \R^E_{\geq 0 }$ is  called {\it admissible} for $\Ga$, if for all $\ga\in \Ga$,
$ \ell_{\rho}(\ga) \geq 1 .$
The {\it admissible set} $\Adm(\Ga)$ of $\Ga$ is defined as the set of all admissible densities for $\Ga$,
\begin{equation}\label{eq:adm-set}
	\Adm(\Ga) := \left\{ \rho \in \R^E_{\geq 0 }: \cN\rho \geq \one \right\}. 
\end{equation} 
Fix $1 \leq p < \infty$, the {\it energy} of the density  $\rho$ is defined as \[\cE_{p,\si}(\rho):=\sum\limits_{e \in E}\si(e)\rho(e)^p.\]
When $p =\infty $,
\[\cE_{\infty,\si}(\rho):= \max\limits_{e\in E} \left\{ \si (e)\rho(e)\right\}.\]
\begin{definition}\label{def:mod}
	The {\it $p$-modulus} of $\Gamma$ is
	\begin{equation}\label{modp}
		\Mod_{p,\sigma}(\Gamma):= 
		\inf\limits_{\rho \in \Adm(\Gamma)} 
		\mathcal{E}_{p,\sigma}(\rho).
	\end{equation}
\end{definition}

When $\si$ is the vector of all ones, we omit $\si$ and write $\cE_{p}(\rho) := \cE_{p,\si}(\rho)$ and $\Mod_{p}(\Ga) :=\Mod_{p,\si}(\Ga)$.
We will routinely identify $\Ga$ with the set of its usage vectors $\left\{ \cN(\ga,\cdot )^T: \ga \in \Ga \right\}$ in $\R^E_{\geq 0}$, hence we can write $\Ga \subset \R^E_{\geq 0 }$. The {\it dominant} of $\Ga$ is defined as
\[ \Dom(\Ga):= \co(\Ga) + \R^E_{\geq 0},\]
where $\co(\Ga)$ denotes the convex hull of $\Ga$ in $\R^E$. We recall Fulkerson duality for modulus.
\begin{definition}\label{def:blocker} Let $\Ga$ be a family of objects on a finite ground set $E$.
	The {\it Fulkerson blocker family} $\widehat{\Ga}$ of $\Ga$ is defined as the set of all the extreme points of $\Adm(\Ga)$:
	\[ \widehat{\Ga} := \Ext\left(\Adm(\Ga)\right) \subset  \R^E_{\geq 0}.\]
\end{definition}
Fulkerson blocker duality \cite{fulkersonblocking} states that
\begin{equation}\label{eq:dom-adm-block-hat}
	\Dom(\widehat{\Ga})= \Adm(\Ga) 
\end{equation}
\begin{equation}\label{eq:dom-adm-block}
	\Dom(\Ga)= \Adm(\widehat{\Ga}) 
\end{equation}
Moreover,  $\widehat{\Ga}$  has its own Fulkerson blocker family, and $\widehat{\widehat{\Ga}}  \subset \Ga$.

When  $1<p < \infty$, let $q:=p/(p-1)$ be the Hölder conjugate exponent of $p$. For any set of weights $\si \in \R^E_{>0}$,  define the dual set of weights $ \widehat{\si}$ as $ \widehat{\si}(e):=\si(e)^{-\frac{q}{p}}$ for all $e\in E$. Let $\Gahat$ be the Fulkerson blocker family of $\Ga$. Fulkerson duality for modulus  \cite[Theorem 3.7]{pietroblocking} states that
\begin{equation}
	\Mod_{p,\si}(\Ga)^{\frac{1}{p}}\Mod_{q, \widehat{\si}}( \widehat{\Ga})^{\frac{1}{q}}=1.
\end{equation}
Moreover, the optimal $\rho^*$ of $\Mod_{p,\si}(\Ga) $ and the optimal $\eta^*$ of $\Mod_{q, \widehat{\si}}( \widehat{\Ga})$ always exist, are unique, and are related as follows:
\begin{equation}\label{eq:weighted-eta-rho}
	\eta^{\ast}(e) = \frac{\si(e)\rho^{\ast}(e)^{p-1}}{\Mod_{p,\si}(\Ga)}, \quad \forall e\in E.
\end{equation}
When $p=1$, we have
\begin{equation}\label{eq:dual-infty}
	\Mod_{1,\si}(\Ga)\Mod_{\infty, \si^{-1}}( \widehat{\Ga})=1.
\end{equation}
Also when $p=2$, we have
\begin{equation}\label{eq:mod2}
	\Mod_{2,\si}(\Ga)\Mod_{2,\si^{-1}}( \widehat{\Ga})=1 \qquad\text{and}\qquad   \displaystyle \eta^{\ast}(e) = \frac{\si(e)}{\Mod_{2,\si}(\Ga)}\rho^{\ast}(e) \quad \forall  e\in E.
\end{equation}
Next, note that \begin{equation}\label{eq:adm-dom}
	\Adm(\Gahat) = \Dom(\Ga) = \co(\Ga) + \R^E_{\geq 0}.
\end{equation}
Hence, we can express

\begin{equation}\label{eq:modmeo}
	\Mod_{2,\sigma}(\Ga)^{-1} = \min\limits_{\eta \in \co(\Ga)} \sum\limits_{e \in E} \si^{-1}(e)\eta^2(e),
\end{equation}
Let $\eta^*$  be the unique optimal density of $\Mod_{2,\si^{-1}}( \widehat{\Ga})$ (which is equivalent to the right-hand side of (\ref{eq:modmeo}).  Let $\cP(\Ga)$ be the set of all probability mass functions in $\R^\Ga_{\geq 0}$ associated with $\Ga$. Any pmf $\mu\in \cP(\Ga)$ such that $\cN^T\mu=\eta^*$ is call an  $\eta^*$-induced pmf.

\section{Appendix: Serial rule for the $\MKL$ problem}\label{sec:serial}

Let us recall some notations in a general setting. Let $E$ be a finite ground set, and let $\Ga$ be a collection of subsets of $E$. Let $\mathcal{N}$ be the matrix whose rows are the indicator vectors of subsets in $\Ga$.
Given a set $A \subset E$, let $\psi_A$ be the restriction operator defined as
\begin{align*}
	\psi_A: 2^E  & \rightarrow 2^A\\
	\ga \subseteq E & \mapsto \ga \cap A.
\end{align*}
Then, for each set $A \subset E$, $\psi_A$ induces a family of objects $ \psi_{A}(\Ga) = \lbrace \ga\cap A: \ga \in \Ga \rbrace.$
Let $\left\{ E_1,E_2 \right\} $ be a partition of the edge set $E$. For each $i = 1, 2$, we define an induced family of objects $\Ga_i :=  \psi_{E_i}(\Ga)$.
Next, we say that a partition $\left\{ E_1,E_2 \right\} $ of the edge set $E$ {\it divides $\Ga$}, if 
$\Ga$ coincides with the {\it concatenation}
\begin{equation}\label{eq:concatenation}
	\Ga_1 \oplus \Ga_2:= \left\{ \ga_1 \cup \ga_2 : \ga_i \in  \Ga_i, i =1,2 \right\}.
\end{equation}
Given a partition $\left\{ E_1,E_2\right\}$ that divides $\Ga$ and a pmf $\mu \in \cP(\Ga)$. For each $i = 1, 2$, we define the {\it marginal} $\mu_i \in \cP(\Ga_i)$ as follows,
\[\mu_{i}(\zeta) :=  \sum \left\{   \mu(\zeta) : \ga \in \Ga, \psi_i(\ga)= \zeta \right\} \quad \forall \zeta \in \Ga_i.\]
On the other hand, given measures $\nu_i \in \cP ( \Ga_i)$ for $i =1,2$, we define their {\it product measure} in $\cP(\Ga)$ as follows, 
\begin{equation}\label{eq:oplus}
	\left( \nu_1  \oplus \nu_2  \right) (\ga)  := \nu_1(\zeta_1) \nu_2(\zeta_2),
\end{equation}
for all $ \displaystyle \ga = \zeta_1 \cup \zeta_2$ where $ \zeta_i \in \Ga_i$, $i =1,2.$ 

With these notations, the  $\MKL_{\si}$ problem defined on the family $\Ga$ with weights $\si \in \R^E_{>0}$ splits into two smaller subproblems as follows.
\begin{theorem}\label{thm:serialMKL}
	Let $\Ga$ be a family of subsets of the ground set $E$ with weights $\si \in \R^E_{>0}$. Let $ E = E_1 \cup E_2$ be a partition that divides $\Ga$. Let $\Ga_1$ and $\Ga_2$ be the family induced by the restriction operators $\psi_{E_1}$ and  $\psi_{E_2}$. Then:
	
	\bi
	\item[(i)] We have \[ \MKL_{\si}(\Ga) = \MKL_{\si}(\Ga_1) +\MKL_{\si}(\Ga_2);\]
	\item[ (ii)] A pmf $\mu \in \cP(\Ga)$  is optimal for $\MKL_{\si}(\Ga)$ if and only if its marginal pmfs $\mu_i \in \cP(\Ga_i), i=1,2$, are optimal for 
	$\MKL_{\si}(\Ga_i)$ respectively;
	\item[ (iii)] Conversely, given  pmfs $\nu_i \in \cP(\Ga_i)$ that are optimal for $\MKL_{\si}(\Ga_i)$  for $i=1,2$, then $\nu_1 \oplus \nu_2 $ is an optimal pmf in $\cP(\Ga) $ for $\MKL_{\si}(\Ga)$; 
	\item[ (iv)] For any pmf $\mu$ with marginals $\mu_i$, if $ e\in E_i$, $i =1,2$, then
	\[ \sum\limits_{\ga\in \Ga:e\in \ga} \mu(\ga)= \sum\limits_{\ga_i\in \Ga_i:e\in \ga_i} \mu(\ga_i).\]
	\ei
\end{theorem}
\begin{proof}
	Let $\mu$ be a pmf in $\cP(\Ga)$ and let $\eta(e) = \sum\limits_{\ga\in \Ga:e\in \ga}\mu(\ga)$. We denote $\underline{\ga}$ a random subset in $\Ga$ distributed by $\mu$.
	Since the partition $\{ E_1,E_2\}$ divides $\Gamma$, we see that for any 
	$\gamma, \gamma' \in \Gamma$, there exist $\gamma_i, \gamma'_i \in \Gamma_i$ 
	for $i = 1, 2$, such that 
	\[
	\gamma = \gamma_1 \cup \gamma_2 \quad \text{and} \quad 
	\gamma' = \gamma'_1 \cup \gamma'_2.
	\]
	For $e \in E_1$,
	\begin{align}\label{eq:bP}
		\bP_{\mu}(e \in \underline{\ga}) &= \sum\limits_{\ga \in \Ga}\cN(\ga,e)\mu(\ga) = 
		\sum_{\gamma_1, \gamma_2} \cN(\ga_1,e)\mu(\gamma_1 \cup \gamma_2) 
		& = 	\sum_{\gamma_1} \cN(\ga_1,e)\mu_1(\gamma_1) = \bP_{\mu_1}(e \in \underline{\ga_1}).
	\end{align}
	The case $e \in E_2$ yields a similar expression with $\mu_1$ replaced by the second marginal $\mu_2$.
	This implies that, if $\mu \in P(\Gamma)$, and $\mu_i$, $i = 1, 2$, are its marginals. Then 
	\[-\sum_{e\in E}\si(e)\log	\bP_{\mu}(e \in \underline{\ga}) = -\sum_{e\in E_1}\si(e)\log	\bP_{\mu_1}(e \in \underline{\ga_1})  -\sum_{e\in E_2}\si(e)\log	\bP_{\mu_2}(e \in \underline{\ga_2}).\]
	If we choose $\mu$ to be an optimal pmf for $\MKL_{\si}(\Gamma)$, we get that
	\begin{align}
		\MKL_{\si}(\Gamma_1) + \MKL_{\si}(\Gamma_2) 
		&\leq -\sum_{e\in E_1}\si(e)\log	\bP_{\mu_1}(e \in \underline{\ga_1}) + -\sum_{e\in E_2}\si(e)\log	\bP_{\mu_2}(e \in \underline{\ga_2}) \\
		&=-\sum_{e\in E}\si(e)\log	\bP_{\mu}(e \in \underline{\ga}) = \MKL_{\si}(\Gamma).
	\end{align}
	
	Conversely, suppose $\nu_i \in P(\Gamma_i)$, $i = 1, 2$, are optimal for their respective $\MKL$ problems. 
	Let $\nu := \nu_1 \oplus \nu_2$ be defined as in (\ref{eq:oplus}). 

	Since the partition $\{E_1, E_2\}$ divides $\Gamma$, this implies that the support set of $\nu$ is contained in $\Gamma$. 
	Moreover, 
	\[
	\sum_{\gamma_1 \in \Gamma_1} \sum_{\gamma_2 \in \Gamma_2} \nu(\gamma_1 \cup \gamma_2) 
	= \sum_{\gamma_1 \in \Gamma_1} \sum_{\gamma_2 \in \Gamma_2} \nu_1(\gamma_1) \nu_2(\gamma_2) 
	= \left( \sum_{\gamma_1 \in \Gamma_1} \nu_1(\gamma_1) \right) 
	\left( \sum_{\gamma_2 \in \Gamma_2} \nu_2(\gamma_2) \right) 
	= 1.
	\]
	Thus, $\nu$ is a pmf in $P(\Gamma)$. Furthermore, 
	\[
	\sum_{\gamma_2 \in \Gamma_2} \nu(\gamma_1 \cup \gamma_2) 
	= \nu_1(\gamma_1) \sum_{\gamma_2 \in \Gamma_2} \nu_2(\gamma_2) 
	= \nu_1(\gamma_1),
	\]
	which shows that $\nu_1$ is a marginal of $\nu$. A similar argument applies to $\nu_2$. Then,  we have 
	\begin{align}
		\MKL_{\si}(\Gamma) &\leq -\sum_{e\in E}\si(e)\log	\bP_{\nu}(e \in \underline{\ga})\\
		&=  -\sum_{e\in E_1}\si(e)\log	\bP_{\nu_1}(e \in \underline{\ga_1}) + -\sum_{e\in E_2}\si(e)\log	\bP_{\nu_2}(e \in \underline{\ga_2})\\
		&= \MKL_{\si}(\Gamma_1) + \MKL_{\si}(\Gamma_2) .
	\end{align}
	This proves parts 1 and 3.
	
	For part 2, assume $\mu \in P(\Gamma)$ is optimal. By part 1, we obtain 
	\[
	0 = 
	\left(-\sum_{e\in E_1}\si(e)\log	\bP_{\mu_1}(e \in \underline{\ga_1}) - \MKL_{\si}(\Gamma_1) \right) 
	+ 
	\left( -\sum_{e\in E_2}\si(e)\log	\bP_{\mu_2}(e \in \underline{\ga_2}) - \MKL_{\si}(\Gamma_2) \right).
	\]
	Since both terms in parentheses are non-negative, they must each equal zero. 
	This implies that the marginals of $\mu$ are also optimal for their respective $\mathrm{MKL}$ problems. 
	
	Conversely, suppose that $\mu$ is a pmf on $\Gamma$ such that both of its marginals $\mu_i$ are optimal for their respective $\mathrm{MKL}$ problems. 
	Then, it follows that $\mu$ is optimal for $\mathrm{MKL}_{\si}(\Ga)$. This completes the proof of part 2.
	
	Part 4 is obtained from equation (\ref{eq:bP}).
\end{proof}
Next, even when a partition does not divide the family, the serial rule still provides a bound.
\begin{corollary}
	Let $\Gamma \subset 2^E$ be a family of objects on a ground set $E$, let $E = E_1 \cup E_2$ be a partition of $E$. Let $\Ga_1$ and $\Ga_2$ be the family induced by the restriction operators $\psi_{E_1}$ and  $\psi_{E_2}$. 
	Then 
	\[
	\MKL_{\si}(\Gamma) \geq \MKL_{\si}(\Gamma_1) + \MKL_{\si}(\Gamma_2),\]
	with equality if the partition divides $\Gamma$.
\end{corollary}
\begin{proof}
	This follows from Theorem \ref{thm:serialMKL} and the following monotonicity property of the $\MKL$ problem. 
	Given an arbitrary partition of $E$, it always holds that $\Gamma \subseteq \Gamma_1 \oplus \Gamma_2$. 
	This implies that $\co(\Gamma) \subseteq \co(\Gamma_1 \oplus \Gamma_2)$, and therefore,
	\[
	\MKL_{\si}(\Gamma) \geq \MKL_{\si}(\Gamma_1 \oplus \Gamma_2).
	\]
\end{proof}
\begin{remark}
	The serial rule in Theorem \ref{thm:serialMKL} can be generalized to partitions with a larger number of parts.
	
\end{remark}
\begin{remark}
	Theorem \ref{thm:serialMKL} can be applied to the direct sum of two matroids where the direct sum of two matroids \( M_1 = (E_1, \cI_1) \) and \( M_2 = (E_2, \cI_2) \) with disjoint ground sets is the matroid
	\begin{equation}\label{eq:directsum}
		M_1 \oplus M_2 := (E_1 \cup E_2, \{ I_1 \cup I_2 : I_1 \in \cI_1, I_2 \in \cI_2 \}).
	\end{equation}
	Hence, the theorem can also be applied to a matroid with at least two connectivity components since it is equal to the direct sum of its connectivity components.
\end{remark}

\end{document}